\documentclass[12pt, a4]{amsart}
\usepackage{amsgen,amsmath,amstext,amsbsy,amsopn,amsfonts,amssymb}
\usepackage{amsmath,amssymb,amsfonts,amsthm,enumerate}
\usepackage{amsthm}
\usepackage[utf8]{inputenc}
\usepackage{color}
\usepackage[pagebackref]{hyperref}
\usepackage[alphabetic]{amsrefs}
\usepackage{xcolor}
\usepackage{tikz-cd}

\usepackage{rotating}

\newtheorem{theorem}{Theorem}

\newtheorem{corollary}[theorem]{Corollary}
\newtheorem{proposition}[theorem]{Proposition}
\newtheorem{obs}[theorem]{Observation} \newtheorem{defi}[theorem]{Definition}

\newenvironment{definition}{\begin{defi}\rm}{\end{defi}}
\newtheorem{exa}[theorem]{Example}

\newtheorem{rem}[theorem]{Remark}

\newtheorem{rems}[theorem]{Remarks}

\newtheorem{ack}[theorem]{Acknowlegment}

\def\NN{{\mathbf N}}
\def\ZZ{{\mathbf Z}}
\def\CCC{{\mathbf C}}
\def\RRR{{\mathbf R}}
\def\QQ{\mathbf Q}

\def\AA{{\mathbf A}}
\def\RR+{{\mathbf R}^*}
\def\TT{\mathbf T}

\def\GG{\mathbf G}

\def\Q_p{{\mathbf Q}_p}
\def\SS{\mathbf S}
\def\TT{\mathbf T}
\def\HH{\mathbf H}

\def\eps{\varepsilon}
\def\Ga{\Gamma}
\def\ga{\gamma}
\def\La{\Lambda}
\def\la{\lambda}

\def\tous{\qquad\text{for all}\quad}

\def\Aut{{\rm Aut}}

\def\Ker{{\rm Ker}}

\def\Sym{{\rm Sym}}

\newcommand{\Rep}{\operatorname{Rep}}

\newcommand{\Ind}{\operatorname{Ind}}

\newcommand{\Bohr}{\operatorname{Bohr}}
\newcommand{\Prof}{\operatorname{Prof}}
\newcommand{\Cong}{\operatorname{Cong}}

\def\tout{\qquad\text{for all}\quad}

\newcommand{\UU}{\mathbf U}
\newcommand{\LL}{\mathbf L}

\begin{document}

\title[The Bohr compactification of an arithmetic  group]{The Bohr compactification of an arithmetic group}

\address{Bachir Bekka \\ Univ Rennes \\ CNRS, IRMAR--UMR 6625\\
Campus Beaulieu\\ F-35042  Rennes Cedex\\
 France}
\email{bachir.bekka@univ-rennes1.fr}

\author{Bachir Bekka}

\thanks{The author acknowledges the support  by the ANR (French Agence Nationale de la Recherche)
through the project Labex Lebesgue (ANR-11-LABX-0020-01) .}
\begin{abstract}
Given a group $\Gamma,$ its Bohr compactification $\Bohr(\Gamma)$ and its profinite completion 
$\Prof(\Gamma)$ are compact groups naturally associated  to $\Gamma$; moreover, $\Prof(\Gamma)$ can be identified
with the quotient   of  $\Bohr(\Gamma)$ by its  connected component $\Bohr(\Gamma)_0.$
We study the structure of $\Bohr(\Ga)$ for  an arithmetic subgroup $\Ga$ of an  algebraic group $\GG$ over $\QQ$. When $\GG$ is unipotent, we  show that $\Bohr(\Ga)$ can be identified with the direct product 
$\Bohr(\Gamma^{\rm Ab})_0\times \Prof(\Ga)$, where $\Gamma^{\rm Ab}= \Gamma/[\Gamma, \Gamma]$
is the abelianization of $\Gamma.$
In the general case, using a Levi decomposition $\GG= \UU\rtimes  \HH$
(where $\UU$ is unipotent and $\HH$ is reductive), we show that $\Bohr(\Ga)$ can be described as the semi-direct product 
of  a certain quotient of $\Bohr(\Ga\cap \UU)$ with $\Bohr(\Ga \cap \HH)$.
When $\GG$ is simple and has higher $\RRR$-rank, $\Bohr(\Ga)$ is isomorphic,
up to a finite group, to the product $K\times \Prof(\Gamma),$
where $K$ is the maximal compact factor of $\GG(\RRR).$
\end{abstract}
\subjclass[2000]{22D10; 22C05; 20E18}
\maketitle
\section{Introduction}
Given a  topological group $G,$ the   \textbf{Bohr compactification} of $G$ is  a pair   $(\Bohr(G), \beta) $ consisting of a 
 compact (Hausdorff)  group $\Bohr(G)$ and a continuous homomorphism $\beta: G \to \Bohr(G) $ with dense image,
satisfying the following universal property: for every compact group $K$
and every continuous homomorphism $\alpha \, \colon G \to K$,
there exists a continuous homomorphism $\alpha' \, \colon  \Bohr(G) \to K$
such that the diagram 
\[
\begin{tikzcd}
& \Bohr(G) \arrow[dashed]{d}{\alpha'} \\
G \arrow{ur}{\beta} \arrow{r}{\alpha} & K
\end{tikzcd}
\]
commutes.
The  pair $(\Bohr(G), \beta) $ is unique in the following sense:
if  $(K', \beta')$ is a pair consisting of a compact group $K'$
and a continuous homomorphism $\beta': G \to K'$ with dense image
satisfying the same universal property (such a pair will be called a Bohr compactification of $G$), then there exists an isomorphism $\alpha: \Bohr(G) \to K'$ of topological groups
such that $\beta' = \alpha \circ \beta$.

 The compact group $\Bohr(G)$  was  first  introduced by A. Weil  (\cite[Chap.VII]{Weil}) as a tool 
for  the study of  almost periodic functions on $G$, a subject initiated  by H. Bohr (\cite{Bohr1}, \cite{Bohr2})
in the case   $G=\RRR$ and  generalized to other groups by J. von Neumann (\cite{VN}) among others.  For more on this subject, see  \cite[\S 16]{Dixm--C*} or  \cite[4.C]{BH}).

The group $\Bohr(\Ga)$ has been determined for only very few non abelian \emph{discrete} groups $\Ga$
(for some general results, see \cite{Hart-Kunen} and \cite{Holm}; for the well-known case of abelian groups, see 
\cite{Anzai-Kakutani1} and Section~\ref{Prop-BohrAbelian}).

In contrast, there is a second much more studied completion of $\Ga$, namely  the \textbf{profinite completion}  of $\Ga$, which is  a pair $(\Prof(\Ga), \alpha)$ consisting of a profinite  group (that is, a projective limit of finite groups)
$\Prof(\Ga)$ satisfying a similar universal property with respect to such groups, together with a homomorphism with $\alpha \, \colon \Gamma \to \Prof(\Ga)$ with dense image.
The  group  $\Prof(\Ga)$  can be realized as the projective limit
   $\varprojlim \Gamma/H$, where $H$ runs over the  family of the normal subgroups of finite index of $\Ga.$
For all this, see  \cite{Ribes-Zal}.
 
The universal property of $\Bohr(\Ga)$ gives rise to a continuous epimorphism
$\alpha' \, \colon  \Bohr(\Ga) \to \Prof(\Gamma).$
 It is easy to see (see Proposition~\ref{Prop-Bohr-Prof} below) that the kernel of $\alpha'$ is 
$\Bohr(\Ga)_0$, the connected component of $\Bohr(\Ga);$ so, we have a short exact sequence
\[
\begin{tikzcd} 
1\arrow{r}&\Bohr(\Ga)_0\arrow{r}&\Bohr(\Ga)\arrow{r}&\Prof(\Ga)\arrow{r} &1. 
\end{tikzcd}
\]

In this paper, we will deal with the case where  $\Ga$ is an arithmetic subgroup  in a   linear algebraic group.
The setting is as follows. Let $\GG$  be a connected linear  algebraic group over $\QQ$ with a fixed  faithful representation 
$\rho: \GG\to GL_m.$
We consider the subgroup $\GG(\ZZ)$ 
 of the group $\GG(\QQ)$ of $\QQ$-points of $\GG$, that is, 
 $$\GG(\ZZ)= \rho^{-1}\left( \rho(\GG) \cap GL_m(\ZZ)\right).$$
  A subgroup $\Ga$ of $\GG(\QQ)$ is called an \textbf{arithmetic subgroup}
   if $\Ga$ is commensurable to $\GG(\ZZ)$, that is, 
 $\Ga\cap \GG(\ZZ)$ has finite index in both $\Ga$ and $\GG(\ZZ)$.
 Observe that $\Ga$ is a discrete subgroup of  the real Lie group
 $\GG(\RRR).$   
  
 We first deal with the case where $\GG$ is unipotent.
 More generally, we  describe the Bohr compactification of any finitely generated nilpotent group.  
 Observe that an arithmetic subgroup in a unipotent algebraic
 $\QQ$-group   is finitely generated (see Corollary 2 of Theorem 2.10 in \cite{Raghunathan}).

   For two topological groups $H$ and $L,$ we write $H\cong L$ if   $H$ and $L$ are topologically isomorphic.
We observe that, when $\Delta$ is   a finitely generated abelian group, 
$\Bohr(\Delta)$ splits as a direct sum $\Bohr(\Delta)= \Bohr(\Delta)_0 \oplus \Prof(\Delta)$; see Proposition~\ref{Prop-BohrAbelian}.
 \begin{theorem}
 \label{Theo-Nilpotent}
 Let $\Ga$ be a finitely generated nilpotent group.
 We have a direct product decomposition 
 $$\Bohr(\Ga)\cong \Bohr(\Ga^{\rm Ab})_0 \times \Prof(\Ga),$$ 
 where $\Ga^{\rm Ab}=\Ga/ [\Ga, \Ga]$   is the abelianization of $\Ga.$
  This isomorphism is induced by the  natural   maps $\Ga\to \Bohr(\Ga^{\rm Ab})$
 and $ \Ga\to \Prof(\Ga),$ together with the projection
$ \Bohr(\Ga^{\rm Ab})\to  \Bohr(\Ga^{\rm Ab})_0 $.
 \end{theorem}
  A crucial tool   in the proof of Theorem~\ref{Theo-Nilpotent} is  the fact  that elements 
 in the commutator subgroup $[\Ga, \Ga]$ of a nilpotent group $\Ga$ are distorted (see Proposition~\ref{Prop-NilDist}).

  We now turn to the case of a general algebraic group $\GG$ over $\QQ.$ 
   Let  $\UU$ be the  unipotent radical   of $\GG$. Then $\UU$ is defined over $\QQ$ and 
there exists a connected reductive $\QQ$-subgroup $\HH$ such that
we have  a \textbf{Levi decomposition} as semi-direct product $\GG= \UU \rtimes \HH$  (see \cite{Mostow}). 
 
 The group $\Lambda=\HH(\ZZ)$ acts by automorphisms on $\Delta=\UU(\ZZ)$
 and hence on $\Bohr(\Delta),$ by the universal property of $\Bohr(\Delta).$
  In general, this action does not extend to an action of  $\Bohr(\Lambda)$ on $\Bohr(\Delta).$
  However,  as we will see below (proof of Theorem~\ref{Theo-General}),
   $\Bohr(\Lambda)$ acts naturally by automorphisms on an appropriate quotient of  $\Bohr(\Delta).$
   
   Observe that (see \cite[Corollary 4.6]{Borel-HC}) every arithmetic subgroup of $\GG(\QQ)$ is commensurable to 
   $\Delta(\ZZ)\rtimes \HH(\ZZ).$ 
   Recall that two topological groups  $G_1$ and $G_2$  are (abstractly) commensurable if
there exist finite index subgroups $H_1$ and $H_2$ of $G_1$ and $G_2$
such that  $H_1$ is topologically isomorphic to $H_1.$ 
If this is the case, then  $\Bohr(G_1)$ and $\Bohr(G_2)$ are commensurable;
in fact,  each one of the  groups  $\Bohr(G_1)$ or $\Bohr(G_2)$
can be described in terms of the other (see Propositions~\ref{Prop-Bohr-SubFinIndex} and \ref{Prop-Bohr-SubFinIndex2} ).
For this reason, we will often  deal with only one chosen
representative of the commensurability class of an arithmetic group.

   \begin{theorem}
 \label{Theo-General}
  Let $\GG$  be a connected linear  algebraic group over $\QQ,$ with Levi decomposition $\GG= \UU\rtimes \HH.$ 
  Set $\Lambda:=\HH(\ZZ), \Delta:=\UU(\ZZ),$ and 
  $\Ga:= \Delta\rtimes \Lambda.$
  Let $ \widehat{\Delta^{\rm Ab}}_{\Lambda-\rm{fin}}$
  be the subgroup of the dual group $\widehat{\Delta^{\rm Ab}}$ of $\Delta^{\rm Ab}$
  consisting of the characters with finite $\Lambda$-orbit.
  We have a semi-direct decomposition $$\Bohr(\Ga)\cong (Q \times \Prof( \Delta)) \rtimes \Bohr(\Lambda),
 $$
 where $Q$ is the   connected component of 
 $\Bohr(\Delta^{\rm Ab})/N$ and $N$ is the annihilator of $\widehat{\Delta^{\rm Ab}}_{\Lambda-\rm{fin}} $ in  $\Bohr(\Delta^{\rm Ab})$.
 This isomorphism is induced by the natural homomophisms $\Delta \to \Bohr(\Delta^{\rm Ab})/N $
 and $\Lambda\to \Bohr(\Lambda).$
 \end{theorem}
 
Theorems ~\ref{Theo-Nilpotent} and  ~\ref{Theo-General} reduce the determination of $\Bohr(\Ga)$ for an arithmetic group $\Ga$ 
in $\GG$ to the  case where $\GG$ is reductive.  We  have a further reduction to the case where
$\GG$ is simply connected and almost simple. 
 Indeed, recall that a  group $L$ is the  \textbf{almost direct product} of subgroups  $L_1, \dots, L_n$   if the product map $L_1\times\cdots\times L_n\to  L$  is a surjective homomorphism with  finite kernel.

Let $\GG$  be a connected  reductive  algebraic group over $\QQ.$
  The commutator subgroup $\LL:=[\GG, \GG]$ of $\GG$ is a connected semi-simple $\QQ$-group and 
  $\GG$ is an almost direct product $\GG=\TT \LL$  for  a central $\QQ$-torus $\TT$
  (see  (14.2) and (18.2) in \cite{Borel-Book})
  Moreover,  $\LL$ is an almost direct product $\LL=\LL_1\cdots\LL_n$
 of connected almost $\QQ$-simple $\QQ$-subgroups $\LL_i,$ called the almost $\QQ$-simple factors of $\LL$
  (see \cite[(22.10)]{Borel-Book}).
  For every $i\in\{1,\dots, n\},$ let   $\widetilde\LL_i$ be  the simply connected covering group $\LL_i.$
  Set $\widetilde{\GG}=  \TT\times \widetilde{\LL_1}\times \cdots\times  \widetilde{\LL_n}.$
  Let $\widetilde{\Ga}$ be the arithmetic subgroup 
  $\TT(\ZZ)\times \widetilde{\LL_1}(\ZZ)\times \cdots\times  \widetilde{\LL_n}(\ZZ)$
  in $\widetilde{\GG}(\QQ).$ The image $\Ga$ of  $\widetilde{\Ga}$ under the isogeny  $p:\widetilde{\GG} \to \GG$ is  an arithmetic subgroup of $\GG(\QQ)$ (see  Corollaries 6.4 and 6.11 in \cite{Borel-HC}).
The  map $p:\widetilde{\Ga} \to \Ga$ induces  an isomorphism
$ \Bohr(\Ga)\cong \Bohr(\widetilde{\Ga})/F$, where
$F$ is the finite normal subgroup $F= \widetilde{\beta}(\ker p)$
and $\widetilde{\beta}: \widetilde{\Ga}\to \Bohr(\widetilde{\Ga})$ is the natural map
(see Proposition \ref{Pro-BohrQuotient}).
  
As an easy consequence of Margulis' superrigidity results, we give  a  description of the Bohr compactification of   an arithmetic lattice  in a  simple algebraic $\QQ$-group $\GG$ under a higher rank assumption. Such a description  does not seem possible for arbitrary $\GG$. For instance, the free non abelian group $F_2$ on two generators is an arithmetic lattice in $SL_2(\QQ)$, but we know of no simple description of $\Bohr(F_2).$

     \begin{theorem}
 \label{Theo-Reductive}
   Let $\GG$ be a connected, simply connected, and   almost simple  $\QQ$-group.
   Assume that the real semisimple Lie  group $\GG(\RRR)$ is  not locally isomorphic  to any group of the form $SO(m, 1)\times K$ or $SU(m,1)\times K$ for a compact Lie group $K$.
 Let $\GG_{\rm nc}$ be the  product of the almost $\RRR$-simple factors $\GG_i$ of $\GG$  for which
 $\GG_i(\RRR)$ is non compact. Let  $\Ga \subset \GG(\QQ)$ be an arithmetic subgroup.
   We have a direct product decomposition 
   $$\Bohr (\Ga) \cong \Bohr (\Ga)_0\times \Prof(\Ga)$$
   and an isomorphism
  $$
  \Bohr(\Ga)_0 \cong   \GG(\RRR)/\GG_{\rm nc}(\RRR),
 $$
 induced by the natural maps 
 $\Ga \to \GG(\RRR)/\GG(\RRR)_{\rm nc}$ and $\Ga\to \Prof(\Ga)$.
 \end{theorem}
 A group $\Ga$ as in Theorem~\ref{Theo-Reductive} is an irreducible  lattice
 in the Lie group  $G=\GG(\RRR)$, that is, the homogeneous space $G/\Ga$ carries a $G$-invariant
 probability measure; moreover,  $\Ga$  is cocompact  in $G$ if and only if $\GG$  is anisotropic over $\QQ$ (for all this, see \cite[(7.8), (11.6)]{Borel-HC}).  
 The following corollary is  a direct consequence of Theorem~\ref{Theo-Reductive}  and of the fact 
 that a non cocompact arithmetic lattice in a semisimple Lie group  has  nontrivial unipotent elements (see \cite[(5.5.14)]{Witte}).
 
 \begin{corollary}
  \label{Cor1}
 With the notation as in Theorem~\ref{Theo-Reductive}, assume that $\GG$ is  isotropic over $\QQ.$
For every arithmetic subgroup $\Ga$ of $\GG(\QQ),$  the natural map
$\Bohr(\Ga)\to\Prof(\Ga)$ is an isomorphism.
 \end{corollary}
  As shown in Section~\ref{S: Examples}, it may happen that  $\Bohr(\GG(\ZZ))\cong \Prof(\GG(\ZZ)),$ even when $\GG(\ZZ)$ is  cocompact in $\GG(\RRR).$.

  A general arithmetic lattice  $\Ga$ has a third completion:
  the \textbf{congruence completion} $\Cong(\Ga)$ of  $\Ga$ is the projective limit
   $\varprojlim \Gamma/H$, where $H$ runs over the  family of the  congruence subgroups  of  $\Gamma$;
recall that  a normal subgroup  of $\Ga$ is a congruence subgroup if it contains  the kernel 
of the map $ \GG(\ZZ)\to  \GG(\ZZ/ N \ZZ)$
of  the reduction modulo $N,$ for some integer $N\geq 1.$
There is a natural surjective homomorphism $\pi:\Prof(\Ga)\to \Cong(\Ga)$.
The so-called \textbf{congruence subgroup problem}  asks whether $\pi$  is injective and hence an isomorphism of
topological groups; more generally, one can ask  for a description of the kernel of $\pi.$
This problem has been extensively studied for  arithmetic subgroups (and, more generally, for $S$-arithmetic subgroups)
 in  various  algebraic groups; for instance,  it is known  that 
$\pi$ is an isomorphism when $\Ga= SL_n(\ZZ)$ for $n\geq 3$ or $\Ga= Sp_{2n}(\ZZ)$ for  $n\geq 2$
(see \cite{BMS}); moreover, the same conclusion is true when $\Ga= \TT(\ZZ)$ for a torus $\TT$ (see \cite{Chevalley})
and  when $\Ga=\UU(\ZZ)$ for a unipotent group $\UU$ (see Proposition~\ref{Pro-CSP-Nilp} below).
For more on the congruence subgroup problem, see for instance \cite{Raghunathan-CSP} or \cite[\S 9.5]{Platonov}.

This paper is organized as follows. In Section~\ref{S:GeneralBohr}, we establish
some general facts about the  Bohr compactifications of commensurable groups
and  the relationship between Bohr compactifications  and unitary representations; we
also give an  explicit  description  of the Bohr compactification for a finitely generated abelian group.
In Section~\ref{Proof-Theo-Nilpotent}, we give the proof of
Theorem~\ref{Theo-Nilpotent}. Section~\ref{Proof-Theo-General}
contains the proof of Theorem~\ref{Theo-General}  and Section~\ref{S-Proof-Theo-Reductive}
the proof of Theorem \ref{Theo-Reductive}.
Section~\ref{S: Examples} is devoted to the explicit computation of the Bohr compactification
for various examples of arithmetic groups.

\section{Some preliminaries}
\label{S:GeneralBohr}
\subsection{Bohr compactifications and unitary representations}
\label{SS:Unirep}
Given a topological group $G,$  we will consider  finite dimensional unitary representations
of $G,$ that is, continuous homomorphisms $G\to U(n)$. 
Two such representations are equivalent if they are conjugate by a unitary matrix.
A representation $\pi$ is irreducible if $\CCC^n$ and $\{0\}$ there are only $\pi(G)$-invariant subspaces
of $\CCC^n$. We denote by $\Rep_{\rm fd} (G)$  the set of equivalence classes  of finite dimensional  unitary representations of $G$ and by $\widehat{G}_{\rm fd}$ the subset  of irreducible ones.
Every $\pi\in \Rep_{\rm fd} (G)$ is a direct sum of representations from $\widehat{G}_{\rm fd}$

When $K$ is a compact group, every  irreducible unitary representation of $K$
is finite dimensional and  $\widehat{K}_{\rm fd}=\widehat{K}$ is the unitary dual space of $K.$
By the Peter-Weyl theorem, $\widehat{K}$  separates the points of $K$.

Let  $\beta: G\to H$ be a a continuous homomorphism of topological groups 
$G$ and $H$ with dense image; then $\beta$ induces  \emph{injective}  maps
 $$\widehat{\beta}:\Rep_{\rm fd} (H) \to \Rep_{\rm fd}(G) \qquad \text{and} \qquad
 \widehat{\beta}:\widehat{H}_{\rm fd}\to \widehat{G}_{\rm fd},$$  
 given by $\widehat{\beta}(\pi)= \pi\circ\beta$ for $\pi\in \Rep_{\rm fd} (H).$
The following proposition, which may be considered as well-known,
 is a useful tool for identifying the Bohr compactification  of a group.
 \begin{proposition}
 \label{GeneralBohr}
 Let $G$ be a topological group, $K$ a compact group, and $\beta: G\to K$ a continuous homomorphism 
 with dense image. The following properties are equivalent:
 \begin{itemize}
 \item[(i)] $(K, \beta)$ is a Bohr compactification of $G;$
 \item[(ii)] the induced map $\widehat{\beta}:\widehat{K}\to \widehat{G}_{\rm fd} $
 is surjective;
 \item[(iii)]  the induced map $\widehat{\beta}:\Rep_{\rm fd} (K) \to \Rep_{\rm fd}(G)$
 is surjective.
 \end{itemize}
 \end{proposition}
 \begin{proof}
 Assume that (i) holds and let $\pi: G\to U(n)$ be an irreducible representation of 
 $G$; by the universal property of the Bohr compactification, there 
  exists a continuous homomorphism $\pi': K \to U(n)$
such that $\pi =\widehat{\beta}(\pi')$ and (ii) follows.

Conversely, assume that (ii) holds. Let $L$ be a compact group and  $\alpha: G\to L$  a 
continuous homomorphism with dense image.
Choose a family $\pi_i: L\to U(n_i)$ of representatives
of $\widehat{L}.$
By  the Peter-Weyl theorem, we may identify    $L$ with its image in   $\prod_{i} U(n_i)$ under
the map  $x\mapsto \oplus_{i} \pi_i(x)$
For every $i$, we have $\pi_i\circ \alpha\in \widehat{G}_{\rm fd}$
and hence $\pi_i\circ \alpha= \widehat{\beta} (\pi'_i)= \pi_i' \circ \beta$ for some 
representation $\pi'_i: K\to U(n_i)$ of $K$. Define a continuous homomorphism 
$$\alpha': K\to \prod_{i} U(n_i) \qquad x\mapsto  \oplus_{i} \pi_i'(x).$$
We have $\alpha'\circ \beta= \alpha$ and  hence 
$$\alpha'(K)= \alpha'\left(\overline{\beta(G)}\right) \subset \overline{\alpha(G)}= L.$$ 
So, (i) and (ii) are equivalent.
It is obvious that  (ii) is equivalent to (iii).
 \end{proof}
 
 The profinite completion $(\Prof(G), \alpha) $ of  $G$  may   be similarly characterized in terms of certain unitary representations of $G.$ Recall first that $(\Prof(G), \alpha) $   is  a pair    consisting of a 
 profinite  group $\Prof(G)$ and a continuous homomorphism $\alpha: G \to \Prof(G) $ with dense image,
satisfying the following universal property: for every profinite group $K$ and every continuous homomorphism $f \, \colon G \to K$, there exists a continuous homomorphism $f' \, \colon  \Bohr(G) \to K$
such that the diagram 
\[
\begin{tikzcd}
& \Prof(G) \arrow[dashed]{d}{f'} \\
G \arrow{ur}{\alpha} \arrow{r}{f} & K
\end{tikzcd}
\]
commutes. Recall that the class of profinite groups coincides with the class of totally disconnected compact groups
(see \cite[Proposition 4.C.10]{BH}).

  Denote by $\Rep_{\rm finite}(G)$  the set of equivalence classes  of finite dimensional  unitary representations $\pi$ of $G$ for which $\pi(G)$ is finite; let  $\widehat{G}_{\rm finite}$ be the subset  of irreducible 
  representations from $\Rep_{\rm finite}(G)$.
  
  If   $\alpha: G\to H$ is a continuous homomorphism of topological groups 
$G$ and $H$ with dense image, then $\beta$ induces  \emph{injective}  maps
 $$\widehat{\alpha}:\Rep_{\rm finite} (H) \to \Rep_{\rm finite}(G) \qquad \text{and} \qquad
 \widehat{\alpha}:\widehat{H}_{\rm finite}\to \widehat{G}_{\rm finite}.$$
 Observe that $\widehat{K}= \widehat{K}_{\rm finite}$ if $K$ is a profinite group. 
 (Conversely, it follows from Peter-Weyl theorem that, if $K$ is a compact group with 
 $\widehat{K}= \widehat{K}_{\rm finite}$, then $K$ is profinite.)
 The proof of the following proposition is similar to the proof of Proposition~\ref{GeneralBohr}
 and will be omitted.
   \begin{proposition}
 \label{GeneralProf}
 Let $K$ be a totally disconnected compact  group and $\alpha: G\to K$ a continuous homomorphism 
 with dense image. The following properties are equivalent:
 \begin{itemize}
 \item[(i)] $(K, \alpha)$ is a profinite completion of $G;$
 \item[(ii)] the induced map $\widehat{\alpha}:\widehat{K}\to \widehat{G}_{\rm finite}$
 is surjective;
  \item[(ii)]  the induced map $\widehat{\beta}:\Rep_{\rm finite} (K) \to \Rep_{\rm finite}(G)$
 is surjective.
 \end{itemize}
 \end{proposition}

 The universal property of $\Bohr(G)$ implies that there is a continuous epimorphism
$\alpha' \, \colon  \Bohr(G) \to \Prof(G)$
such that the diagram 
\[
\begin{tikzcd}
& \Bohr(G) \arrow[dashed]{d}{\alpha'} \\
G \arrow{ur}{\beta} \arrow{r}{\alpha} & \Prof(G)
\end{tikzcd}
\]
commutes. We record the following elementary but basic  fact mentioned in the introduction.
\begin{proposition}
\label{Prop-Bohr-Prof}
The kernel of $\alpha':\Bohr(G) \to \Prof(G)$
coincides with  the connected component $\Bohr(G)_0$ of $\Bohr(G).$
\end{proposition}
\begin{proof}
Since $\Bohr(G)_0$ is connected and $\Prof(G)$ is totally disconnected, 
$\Bohr(G)_0$ is contained in $\Ker \alpha'$.
So, $\alpha'$ factorizes to a  continuous epimorphism 
$\alpha'':  K \to \Prof(G),$ where $K:= \Bohr(G)/\Bohr(G)_0$
and we have a commutative diagram 
\[
\begin{tikzcd}
& K \arrow{d}{ \alpha''}\\
G\arrow{ur}{p\circ\beta} \arrow{r}{\alpha} & \Prof(G).
\end{tikzcd}
\]
where $p: \Bohr(G)\to K$  is the canonical epimorphism.
Since $K$ is  a totally disconnected compact group, there exists a continuous epimorphism $f: \Prof(G)\to K$
and we have a commutative diagram
\[
\begin{tikzcd}
&  K\\
G\arrow{ur}{p\circ\beta} \arrow{r}{\alpha} & \Prof(G)\arrow{u}{f}.
\end{tikzcd}
\]
For every $g\in G,$ we have 
$$f(\alpha'' (p\circ \beta(g)))= f(\alpha(g))=p\circ \beta(g);$$
since $p\circ \beta(G)$ is dense in $K,$ it follows that $f\circ \alpha''$
is the identity on $K.$ This implies that $\alpha''$ is injective and hence an isomorphism.
\end{proof}
 \subsection{Bohr compactifications of commensurable groups}
 \label{SS:Commensurable}
Let $G$ be a  topological group and $H$ be a closed subgroup of finite index in $G$.
We first determine $\Bohr(H)$ in terms of $\Bohr(G)$.
 \begin{proposition}
 \label{Prop-Bohr-SubFinIndex} 
 Let $(\Bohr(G), \beta)$ be the Bohr compactification of $G$.
Set $K: =\overline{\beta(H)}.$
  \begin{itemize}
  \item[(i)]  $K$ is a  subgroup of finite index of $\Bohr(G)$.
\item[(ii)] $(K, \beta|_H)$ is a Bohr compactification of $H.$
  \item[(iii)]   $K$ and $\Bohr(G)$ have the same  connected component of the identity.
 \end{itemize}
 \end{proposition}
\begin{proof}
Item (i) is obvious and Item (iii) follows from Item (i). To show Item (ii), let   $\pi$ be a unitary representation
of  $H$ on $\CCC^n.$
Since $H$ has finite index in $H,$ the induced representation  $\rho:= \Ind_{H}^G \pi$, which is a unitary representation of 
$G,$ is finite dimensional. Hence, there exists $\rho'\in \Rep_{\rm fd} (\Bohr(G))$
such that $\rho= \rho'\circ \beta$. Now, $\pi$ is equivalent to a subrepresentation  of  the restriction   
of $\rho$ to $H$ (see \cite[1.F]{BH}); so, we may identify $\pi$ 
with the representation of $H$ defined by a  $\rho(H)$-invariant subspace
$W$ of  the space of $\rho.$  Then $W$ is $\rho'(K)$-invariant and defines therefore
a representation $\pi'$ of $K.$  We have 
$\pi=\pi'\circ (\beta|_H)$ and Proposition~\ref{GeneralBohr} shows that Item (ii)  holds.
\end{proof}

Next, we want to determine $\Bohr(G)$ in terms of $\Bohr(H)$.

Given a  compact group $K$ and a finite set $X,$ we define another compact group,
we call the \textbf{induced group} of $(K, X)$,   as 
$$\Ind (K,X):=K^X\rtimes \Sym(X),$$
where the  group $\Sym (X)$ of  bijections of $X$ 
acts by permutations of indices on $K^X:$
$$
\sigma ((g_x)_{x\in X})= (g_{\sigma^{-1}(x)})_{x\in X} \tout \sigma\in \Sym (X), (g_x)_{x\in X}\in K^X
$$
Observe that, if $\pi: K\to U(n)$ is a representation of $K$ on $V=\CCC^n,$ 
then a unitary representation  $\Ind(\pi)$ of  $\Ind (K,X)$  on on $V^X$ is defined by 
$$\Ind(\pi) ((g_x)_{x\in X}, \sigma)(v_x)_{x\in X}= (\pi(g_x)v_{\sigma^{-1}(x)})_{x\in X},
$$
for $ ((g_x)_{x\in X}, \sigma)\in \Ind (K, X)$ and $ (v_x)_{x\in X}\in V^X.$

Coming back to our setting, where  $H$ is a closed subgroup of finite index in $G$,
we fix a transversal $X$ for the right cosets of $H;$ so, we have  a disjoint union $G = \bigsqcup_{x \in X} Hx$.
For every $g \in G$ and $x \in X$,  let $x \cdot g$ and $c(x, g))$ be the unique elements in $X$ and $H$
such that $xg \, = \, c(x, g) (x \cdot g).$
Observe that 
$$
X \times G \to X, \qquad  (x,g) \mapsto x \cdot g
$$
is an action of $G$ on $X$ (on the right), which is equivalent to the natural action of $G$ on $H \backslash G$ 
given by right multiplication. In particular, for every $g\in G,$   the map 
$\sigma(g):x\mapsto x\cdot g^{-1}$ 
belongs to $ \Sym(X)$ and we have a homomorphism
$$G\mapsto \Sym(X), \,\, g\mapsto \sigma(g).$$

\begin{proposition}
 \label{Prop-Bohr-SubFinIndex2} 
 Let $(\Bohr(H), \beta)$ be the Bohr compactification of $H$.
 Let  $\Ind(\Bohr(H), X)$ be the compact group defined as above.
 Consider the map  $ \widetilde{\beta}: G\to \Ind(\Bohr(H), X)$ defined by 
 $$\widetilde{\beta}(g)= (\beta(c(x,g)))_{x\in X}, \sigma(g)) \tout g\in G.$$
 The closure of  $\widetilde{\beta}(G)$ in $\Ind(\Bohr(H), X),$
 together with the map $\widetilde \beta$,   is a Bohr compactification of $G.$
  \end{proposition}
\begin{proof}
It is readily checked that  $ \widetilde{\beta}: G\to\Ind(\Bohr(H), X)$  is a continuous homomorphism.
Let   $\rho: G\to U(n)$ be a finite dimensional unitary representation of $G.$
Set $\pi:= \rho|_H\in  \Rep_{\rm fd}(H).$
There exists $\pi'\in \Rep_{\rm fd} (\Bohr(H))$ such that $\pi= \pi'\circ \beta$.
Let  $\widetilde{\pi}:= \Ind_H^G \pi .$
As is well-known (see \cite[1.F]{BH}), $\widetilde{\pi}$ can be realized on $V^X$ for $V:= \CCC^n$  by the formula
$$
\widetilde{\pi}(g) (v_x)_{x\in X} )=( \pi (c(x, g))v_{x\cdot g})_{x\in X}=( \pi (c(x, g))v_{\sigma(g^{-1})x})_{x\in X},
$$
for all $g \in G$ and  $(v_x)_{x\in X} \in V^X.$
With the unitary representation $\Ind(\pi')$ of $\Ind(\Bohr(H), X)$ defined as above, we have 
therefore 
$$
\leqno{(*)}\qquad \widetilde{\pi}(g)= \Ind(\pi') (\widetilde{\beta}(g)) \tout  g\in G,
$$
that is, $\widetilde{\pi}=  \Ind(\pi')\circ \widetilde{\beta}.$ Now, 
 $$\widetilde{\pi}=\Ind_H^G \pi= \Ind_H^G (\rho|_H)$$
 is equivalent to the tensor product  representation $\rho \otimes \lambda_{G/H}$,
 where $\lambda_{G/H}$ is the regular representation of $G/H$
 (see \cite[E.2.5]{BHV}). Since $\lambda_{G/H}$ contains the trivial representation of $G,$
 it follows that $\rho$ is equivalent to a subrepresentation of $\widetilde{\pi}$; so, we can identify $\rho$
 with the representation of $G$  defined by a  $\widetilde{\pi}(G)$-invariant  subspace $W$ of $V^X$. 
 Denoting by $L$ the closure of $\widetilde{\beta}(G)$, it follows from $(*)$ that $W$ is invariant under  $\Ind(\pi')(L)$ and so defines a representation $\rho'$ of $L$.
 Then  $\rho=  \rho'\circ \widetilde{\beta}$ and the claim follows from Proposition~\ref{GeneralBohr}.
 \end{proof}
 We will also need the following well-known (see \cite[Lemma 2.2]{Hart-Kunen}) 
 description of the Bohr compactification  of a quotient of $G$ in terms of the Bohr compactification of $G$. 
\begin{proposition}
\label{Pro-BohrQuotient}
Let $(\Bohr(G), \beta)$ be the Bohr compactification of the topological group $G$
and let $N$ be a closed normal subgroup of $G.$
Let $K_N$ be the closure of $\beta(N)$ in $\Bohr(G)$
\begin{itemize}
\item[(i)] $K_N$ is a  normal subgroup of $\Bohr(G)$ and 
$\beta$ induces a continuous homomorphism $\overline{\alpha}: G/N\to  \Bohr(G)/K_N$
\item[(ii)] $(\Bohr(G)/K_N, \overline{\alpha})$ is a Bohr compactification of $G/N.$
\end{itemize}
\end{proposition}
\begin{proof}
Let $(\Bohr(G/N),\overline{\beta})$ be the Bohr compactification of $G/N.$
The canonical homomorphism $\alpha: G\to G/N$ induces a continuous 
homomorphism $\alpha': \Bohr(G)\to \Bohr(G/N)$ 
such that the diagram 
\[
\begin{tikzcd}
G\arrow{r}{\beta}\arrow{d}{\alpha}& \Bohr(G) \arrow[dashed]{d}{\alpha'} \\
G/N  \arrow{r}{\overline{\beta}} & \Bohr(G/N)
\end{tikzcd}
\]
commutes.
It follows that $\beta(N)$ and hence $K_N$ is contained in $\Ker \, \alpha'.$
So, we have  induced homomorphisms $\overline{\alpha}: G/N\to   \Bohr(G)/K_N$
and $\overline{\alpha'}: \Bohr(G)/K_N\to \Bohr(G/N),$
giving rise to a  commutative diagram
\[
\begin{tikzcd}
& \Bohr(G)/K_N \arrow[dashed]{d}{\overline{\alpha'}} \\
G/N  \arrow[dashed]{ur}{\overline{\alpha}}\arrow{r}{\overline{\beta}} & \Bohr(G/N).
\end{tikzcd}
\] 
It follows that $(\Bohr(G)/K_N, \overline{\alpha})$ has the same  universal property 
for $G/N$ as $(\Bohr(G/N),\overline{\beta})$. Since $\overline{\alpha}$
has dense image,  $(\Bohr(G)/K_N, \overline{\alpha})$ is therefore a Bohr compactification of $G/N.$
\end{proof}

\subsection{Bohr compactification of finitely generated abelian groups}
\label{SS-BohrAbelian}
Let $G$ be a locally compact abelian group. Its dual group ${\widehat G}$
consists of the continuous homomorphism from $G$ to the circle
group $\SS^1;$
equipped with the topology of uniform convergence on compact subsets, 
${\widehat G}$ is again a locally compact abelian group. 
 Let  ${\widehat G}_{\rm disc}$ be the group $\widehat G$
equipped with the discrete topology. It is well-known (see e.g. \cite[Proposition 4.C.4]{BH}) that 
the Bohr compactification of $G$ coincides with the 
dual  group $K$ of ${\widehat G}_{\rm disc}$,
together with the embedding $i: G \to K$ given by 
$i(g) (\chi)= \chi(g)$   for all $g \in G$ and $\chi \in {\widehat G}.$
Notice that this implies that, by Pontrjagin duality, the dual group 
of $\Bohr(G)$ coincides with ${\widehat G}_{\rm disc}$.

A more precise information  on the structure of the Bohr compactification is available in the case
of  a (discrete) finitely generated abelian group.
As is well-known, such a group $\Ga$ splits a direct sum $\Ga=F \oplus A$
of a finite  group $F$ (which is its torsion subgroup) and a
free abelian group $A$ of finite rank $k\geq 0$, called the rank of $\Ga$.
Recall that   $\ZZ_p$ denotes  the ring  of $p$-adic integers for a prime $p$
and $\AA$  the  ring of ad\`eles over $\QQ.$
\begin{proposition}
\label{Prop-BohrAbelian} 
Let $\Ga$ be a finitely generated abelian group of rank $k.$
\begin{itemize}
\item[(i)] We have a direct sum decomposition 
$$\Bohr(\Ga) \cong \Bohr(\Ga)_0 \oplus \Prof(\Ga).$$
\item[(ii)] We have 
$$\Prof(\Ga)\cong  {F}\oplus \prod_{p \ \text{\rm prime}} \ZZ_p^k,$$
 where $F$ is a finite group.
  \item [(iii)] We have 
  $$\Bohr(\Ga)_0\cong \prod_{ \omega \in \frak{c}} \AA^k/\QQ^k,$$
  a product of uncountably many copies  of the adelic solenoid $\AA^k/\QQ^k$.
\end{itemize}
\end{proposition}
\begin{proof}
We have $\Ga\cong F \oplus \ZZ^k$ for a finite group $F$ and $\Bohr(\ZZ^k)= \Bohr(\ZZ)^k.$ So, it suffices to determine $\Bohr(\ZZ).$ As mentioned above, $\Bohr(\ZZ)$  can be identified with the dual group 
of the circle $\SS^1$ viewed as discrete group.
Choose a linear basis $\{1\} \cup\{ x_\omega\mid \omega \in \frak{c}\}$
of $\RRR$ over $\QQ$. Then $\SS^1\cong \RRR/\ZZ$
is isomorphic to the abelian group
$$(\QQ/\ZZ) \oplus \oplus_{\omega \in \frak{c}} \QQ.$$
Hence, 
$$
\Bohr(\ZZ)\cong \widehat{\QQ/\ZZ} \oplus \prod_{\omega \in \frak{c}} \widehat{\QQ}.
$$
Now, $$\QQ/\ZZ= \oplus_{p \ \text{prime}} Z(p^{\infty}),$$
 with $Z(p^{\infty})= \varinjlim_{k} \ZZ/ p^k \ZZ$ the $p$-primary
component of $\QQ/\ZZ.$  Hence, 
$$\widehat{Z(p^{\infty})}\cong \varprojlim_k \ZZ/ p^k \ZZ=\ZZ_p.$$
On the other hand, $\widehat{\QQ}$ can be identified with the solenoid $\AA/ \QQ$
(see e.g. \cite[(25.4)]{HewittRoss}). It follows that 
$$ \Bohr(\Ga) \cong\prod_{p \ \text{\rm prime}} \ZZ_p \oplus \prod_{\omega \in \frak{c}} \AA/ \QQ.$$
\end{proof}
\subsection{Restrictions of representations to normal subgroups}
\label{SS:NormalSubg}
Let $\Ga$ be a group and $N$ a normal subgroup  of $\Ga$.
 Recall that $\Ga$ acts on $\widehat{N}_{\rm fd}$: for $\sigma\in \widehat{N}_{\rm fd}$ and $\ga \in \Ga,$ 
the conjugate representation $\sigma^\ga \in \widehat{N}_{\rm fd}$ is defined by 
$$
\sigma^\ga(n) \, = \, \sigma( \ga^{-1} n \ga),
\hskip.5cm \text{for all} \hskip.2cm
n \in N .
$$
The stabilizer $\Ga_\sigma$ of $\sigma$ is the subgroup consisting of all $\ga\in \Ga$ for which $\sigma^\ga$ is equivalent $\sigma;$
observe that $\Ga_\sigma$ contains $N.$
 
 Given a unitary representation $\rho$ of $N$ on a finite dimensional vector space $V$
 and $\sigma\in \widehat{N}_{\rm fd},$
 we denote by $V^\sigma$ the \emph{$\sigma$-isotypical component} of 
 $\rho,$ that is, the sum of all $\rho$-invariant subspaces $W$ 
 for  which the restriction of $\rho$  to $W$ is equivalent to $\sigma.$
 Observe that $V$ decomposes as direct sum 
 $V=\oplus_{\sigma\in\Sigma_\rho} V^\sigma$,
 where $\Sigma_\rho$ is the finite set of  $\sigma\in \widehat{N}_{\rm fd}$ 
 with $V^\sigma\neq \{0\}.$

\begin{proposition}
\label{Prop-MackeyMachine1}
Let  $\pi$ be an irreducible  unitary representation
of $\Ga$ on a finite dimensional vector space $V.$  
Let  $V=\oplus_{\sigma\in\Sigma_{\pi|_N}}V^\sigma$ be the 
decomposition of the restriction $\pi|_N$ of $\pi$ to $N$  into isotypical components.
Then  $ \Sigma_{\pi|_N}$ coincides with a $\Ga$-orbit: there exists
$\sigma\in\widehat{N}_{\rm fd}$ such that 
$ \Sigma_{\pi|_N}= \{\sigma ^\ga\ : \ \ga \in \Ga\};$ in particular, 
$\Ga_\sigma$ has finite index in $\Ga.$
\end{proposition}
\begin{proof}
Let $\sigma\in  \Sigma_{\pi|_N}$ and  fix a transversal $T$ for the left cosets of $\Ga_\sigma$ with $e\in T$.
Then  $V^{\sigma^t}=\pi(t) V^{\sigma}$ for all $t\in T$
Since $\pi$ is irreducible and $\sum_{t\in T} \pi(t) V^{\sigma}$ is $\pi(\Ga)$-invariant,
it follows that  $\Sigma_{\pi|_N}$ is a $\Ga$-orbit.
\end{proof}
 \section{Proof of Theorem~\ref{Theo-Nilpotent}}
\label{Proof-Theo-Nilpotent}
\subsection{Distortion and Bohr compactification}
Let  $\Ga$ be a finitely generated group with a finite set $S$ of generators.
For $\ga\in \Ga$,  denote by $\ell_S(\ga)$  the word length of $\ga$ with respect
to $S\cup S^{-1}$ and set 
$$t(\ga)=\liminf_{n \to \infty} \dfrac{\ell_S(\ga^n)}{n}.$$
The number $t(\ga)$ is called the \emph{translation number} of $\ga$ in \cite{Gersten}
\begin{definition}
\label{Def-Dist}
An element $\ga\in \Ga$  is said to be  \textbf{distorted} if $t(\ga)=0.$
\end{definition}
In fact, since the sequence $n\mapsto \ell_S(\ga^{n})$ is subadditive, we have, by Fekete's lemma,
$$t(\ga)=\lim_{n \to \infty} \dfrac{\ell_S(\ga^n)}{n}= \inf\left\{ \dfrac{\ell_S(\ga^n)}{n}: n\in \NN^*\right\}$$
The property of being distorted is independent of the choice of the set of generators. Distorted elements are called \emph{algebraically parabolic} in \cite[(7.5), p.90]{Gromov}, but we prefer to use the terminology from \cite{Franks}.
The relevance of distorsion  to the Bohr compactification lies in the following proposition;
for a related result with a   similar  proof, see  \cite[(2.4)]{LMR}.

\begin{proposition}
\label{Pro-DistorsionBohr}
Let  $\Ga$ be a finitely generated group and $\ga\in \Ga$ a distorted element.
Then, for  every  finite dimensional unitary representation  $\pi: \Ga\to U(N)$ of $\Ga,$
the matrix   $\pi(\ga)\in U(N)$ has finite order.
\end{proposition}
\begin{proof}
It suffices to show that all eigenvalues of the unitary matrix $\pi(\ga)$ are 
roots of unity.  Assume, by contradiction, that $\pi$ has an eigenvalue
$\la\in \SS^1$ of infinite order.

Let $S$ be a finite set of generators of 
$\Ga$ with $S= S^{-1}.$  
The group $\pi(\Ga)$ is  generated by the set $\{\pi(s)\mid s\in S\}.$
Hence, $\pi(G)$ is contained in $GL_N(L)$, where $L$
is the subfield of $\CCC$ generated by the matrix coefficients of the $\pi(s)$'s.
 It follows that $\la$ is contained in a finitely generated extension $\ell$ of $L$.
 By a lemma of Tits (\cite[Lemma 4.1]{Tits}), there exists a locally compact field $k$ endowed with an
absolute value $| \cdot|$ and a field embedding $\sigma: \ell\to k$  such that  $|\sigma(\la)| \neq 1.$
 Upon replacing $\ga$ by $\ga^{-1}$, we may assume that $|\sigma(\la)|>1.$
 
 Define a function (``norm")  $\xi\mapsto \Vert \xi \Vert$ on $k^N$ by 
 $$\Vert \xi \Vert =\max \{|\xi_1|, \dots, |\xi_N|\} \tout \xi= (\xi_1, \dots, \xi_N)\in k^N.$$
 For a matrix $A\in GL_N(k),$ set $\Vert A\Vert= \sup_{\xi \neq 0} \Vert A\xi\Vert/ \Vert \xi\Vert.$
 It is obvious that  $\Vert A\xi\Vert \leq \Vert A\Vert \Vert \xi\Vert$ for  all $\xi\in k^N$ and hence
 $$\Vert AB\Vert \leq \Vert A\Vert\Vert B\Vert \tous A, B\in GL_N(k). \leqno{(*)}$$
 In particular, we have 
 $\Vert A^n\Vert \leq \Vert A\Vert^n$  for all $A\in GL_N(k)$ and $n\in \NN.$
 
 For a matrix $w\in GL_n(\ell),$ denote by  $\sigma(w)$ the matrix in $GL_n(k)$ obtained by applying $\sigma$
 to the entries of $w.$
 Set $A_s= \sigma(\pi(s))$ for $s\in S$ and $A:= \sigma(\pi(\ga)).$
 With  $$C:= \max \{\Vert A_s\Vert : s\in S\},$$
 it is clear that Inequality $(*)$ implies that
 $$\Vert A^n\Vert =\Vert \sigma(\pi(\ga^n))\Vert \leq  C ^{\ell_S(\ga^n)} \tout n\in \NN. \leqno{(**)}$$
 On the other hand,  $\sigma(\la)$ is an eigenvalue of $A$; so, there exists $\xi \in k^N\setminus \{0\}$ such that
 $A \xi= \sigma(\la)\xi$ and hence $A^n \xi= \sigma(\la)^n \xi$ for all $n\in\NN.$
 So, for every $n\in \NN,$ we have 
 $$\Vert A^n \xi\Vert= |\sigma(\la)|^n\Vert \xi \Vert $$
 and this implies that 
 $$\Vert A^n \Vert \geq  |\sigma(\la)|^n.$$
 In view of $(**)$, we obtain therefore
 $$\dfrac{\ell_S(\ga^n)  \log C}{n} \geq   {\log |\sigma(\la)|} \tout n\in \NN.$$
 Since $|\sigma(\la)|>1,$ this  contradicts the fact that $\liminf_{n \to \infty} \dfrac{\ell_S(\ga^n)}{n}=0.$
 \end{proof}
 \subsection{Distorted elements in nilpotent groups}
Let $\Ga$ be a finitely generated nilpotent subgroup.  For subsets $A, B$ in $\Ga,$
we let $[A,B]$ denote the subgroup of $\Ga$
generated by all commutators $[a, b]=aba^{-1}b^{-1}$, for $a\in A$ and  $b\in B.$ Let
$$ \Ga^{(0)}  \supset \Ga^{(1)} \supset \cdots  \supset \Ga^{(d-1)} \supset  \Ga^{(d)} =\{e\}$$
be the lower central series of $\Ga$, defined inductively by $ \Ga^{(0)} = \Ga$ and  $ \Ga^{(k+1)} = [ \Ga^{(k)} , \Ga].$
 The step of nilpotency of $\Ga$ is the smallest $d\geq 1$ such that $\Ga^{(d-1)}\neq \{e\}$ and  $\Ga^{(d)}=\{e\}$.
\begin{proposition}
\label{Prop-NilDist}
Let $\Ga$ be a finitely generated nilpotent subgroup.
Every $\ga\in  \Ga^{(1)} =[\Ga, \Ga]$ is  distorted.
\end{proposition}
\begin{proof}
Let $S$ be a finite set of generators of  $\Ga$ with $S= S^{-1}.$   Let $d\geq 1$ be the step  of nilpotency of $\Ga$.
The case $d=1$ being trivial, we will assume that $d\geq 2.$
We will show by induction on  $i\in\{1, \dots, d-1\}$ that every $\ga\in  \Ga^{(d-i)}$ is distorded.

\vskip.2cm
$\bullet$ {\it First step.}  Assume that $i=1$.  It is well-known that every element $\ga$ in $\Ga^{(d-1)}$ is distorted
(see for instance \cite[(7.6), p. 91]{Gromov}); in fact,  more precise estimates are  available:
for every $\ga\in  \Ga^{(d-1)},$ we have $\ell_S(\ga^n)=O(n^{1/d})$ as $n\to \infty$ (see  \cite[2.3 Lemme]{Tits2} or \cite[Lemma 14.15]{Drutu}).

\vskip.2cm
$\bullet$ {\it Second step.} Assume  that, for every 
finitely generated nilpotent subgroup $\Lambda$ of step $d'\geq 2$, every element $\delta \in  \Lambda^{(d'-i)}$ is distorded for   $i\in\{1, \dots, d'-2\}.$
Let $\ga\in \Ga^{(d-i-1)}$ and fix $\eps>0.$

 The quotient group $\overline{\Ga}= \Ga/ \Ga^{(d-1)}$ is  nilpotent of step $d'=d-1$
and $p(\ga)\in \overline{\Ga}^{(d'-i)}$, where $p: \Ga\to \overline{\Ga}$ is the quotient map.
By induction hypothesis, $p(\ga)$ is distorted in $\overline{\Ga}$ with respect to the generating set $\overline{S}:=p(S).$
So, we have  $\lim_{n \to \infty} \dfrac{\ell_{\overline{S}}(p(\ga)^n)}{n}=0;$ hence, we can find
an integer $N\geq 1$ such that
$$ \forall n\geq N, \exists \delta_n \in \Ga^{(d-1)} \ : \  \dfrac{\ell_{S}(\ga^n \delta_n)}{n} \leq \eps . \leqno{(*)}$$
By the first step, we have $\lim_{k \to \infty} \dfrac{\ell_S(\delta_N^k)}{k}=0,$
since $\delta_N\in  \Ga^{(d-1)};$ so, there exists $K\geq 1$ such that
$$
 \forall k\geq K \ : \  \dfrac{\ell_{S}(\delta_N^k )}{k} \leq\eps. \leqno{(**)}
 $$
 Let $k\geq K.$  We have
 $$
 \dfrac{\ell_{S}(\ga^{Nk})}{Nk}= \dfrac{\ell_{S}((\ga^{Nk}\delta_N^k)(\delta_N^{-1})^k)}{Nk} \leq  \dfrac{\ell_{S}(\ga^{Nk}\delta_N^k)}{Nk} + \dfrac{\ell_{S}(\delta_N^k)}{Nk}. \leqno{(***)}
$$
Now, since $\Ga^{(d-1)}$ is contained in the center of $\Ga,$ the elements 
$\delta_N$ and $\gamma_N$ commute and hence, by $(*)$, we have
$$
  \dfrac{\ell_{S}(\ga^{Nk}\delta_N^k)}{Nk}=  \dfrac{\ell_{S}((\ga^{N}\delta_N)^k)}{Nk} \leq  k\dfrac{\ell_{S}(\ga^{N}\delta_N
  )}{Nk}= \dfrac{\ell_{S}(\ga^{N}\delta_N
  )}{N}\leq \eps.
 $$
 So, together with $(***)$ and $(**)$, we obtain
 $$
  \forall k\geq K \ : \  \dfrac{\ell_{S}(\ga^{Nk})}{Nk} \leq 2\eps.
  $$
This shows that $t(\ga)=0.$
\end{proof}

\subsection{Congruence subgroups in unipotent groups}
The following result, which shows that the congruence subgroup problem
has a positive solution for unipotent groups, is well-known  (see the sketch in \cite[p.108]{Raghunathan-CSP});
 for the convenience of the reader, we reproduce its short proof.

\begin{proposition}
\label{Pro-CSP-Nilp}
Let $\UU$ be a unipotent algebraic  group over $\QQ.$
Let $\Ga$ be  an arithmetic subgroup of $\UU(\QQ).$
Then every finite index subgroup of $\Ga$ is a congruence subgroup. 
\end{proposition}
\begin{proof}
We can find a sequence 
$$ \UU=\UU_0 \supset \UU_1 \supset \cdots  \supset \UU_{d-1} \supset  \UU_d =\{e\}$$
of normal $\QQ$-subgroups of $\UU$ such that the factor groups  $\UU_i/ \UU_{i+1}$ are $\QQ$-isomorphic
to $\GG_a$, the additive group of dimension 1 (see \cite[(15.5)]{Borel}).

We proceed by induction on $d\geq 1$. If $d=1,$  then $\Ga$ is commensurable to $\ZZ$
and the claim is  obvious true. 
Assume that $d\geq 2.$ Then $\UU$ can be written as semi-direct product $\UU= \UU_1\rtimes \GG_a.$
By \cite[Corollary 4.6]{Borel-HC}, $\Ga$ is commensurable to $\UU_1(\ZZ)\rtimes \ZZ.$
Let $H$ a subgroup of finite index in $\Ga.$ Then 
$H\cap \UU_1(\ZZ)$ has finite index in $\UU_1(\ZZ)$ and hence,  by induction hypothesis,
contains the kernel of the reduction $\UU_1(\ZZ)\to  \UU_1(\ZZ/ N_1 \ZZ)$
modulo some $N_1\geq 1.$ Moreover, $H\cap \ZZ= N_2\ZZ$ for some $N_2\geq 1.$
Hence, $H$ contains  the kernel of the reduction $\UU(\ZZ)\to  \UU(\ZZ/ N_1N_2 \ZZ)$
modulo $N_1N_2.$
\end{proof}

 \subsection{Proof of Theorem~\ref{Theo-Nilpotent}}
  \label{SS:TheoNil}

  Let $\Ga$ be a finitely generated nilpotent group and $\alpha: \Ga\to \Prof(\Ga)$  the canonical homomorphism.
  Recall (see Proposition~\ref{Prop-BohrAbelian}) that  the 
  Bohr compactification of $\Ga^{\rm Ab}=\Ga/ [\Ga, \Ga]$ splits as a direct sum 
  $$\Bohr(\Ga^{\rm Ab})= \Bohr(\Ga^{\rm Ab})_0 \oplus B_1,$$
  for a closed subgroup $B_1\cong \Prof(\Ga^{\rm Ab}).$
  Let  $p:  \Bohr(\Ga^{\rm Ab})\to  \Bohr(\Ga^{\rm Ab})_0 $ be the  corresponding  projection.
Denote by  $\beta_0: \Ga\to \Bohr(\Ga^{\rm Ab})$ the map induced by the quotient homomorphism
 $\Ga \to \Ga^{\rm Ab}$.
Set 
 $$ K:= \Bohr(\Ga^{\rm Ab})_0 \times \Prof(\Ga),$$
 and let $\beta: \Ga \to K$ be the homomorphism $\ga\mapsto (p\circ \beta_0(\ga), \alpha(\ga)).$
 We claim that $(K, \beta)$ is a Bohr compactification for $\Ga.$
 
  \vskip.2cm
  $\bullet$ {\it First step.} We   claim that $\beta(\Ga)$ is dense in $K.$ 
  Indeed, let $L$ be the closure of $\beta(\Ga)$ in $K$ and $L_0$ its connected component.
  Since $\Prof(\Ga)$ is totally disconnected, the projection of $L_0$ on $\Prof(\Ga)$
  is trivial; hence $L_0= K_0\times \{1\}$ for a connected closed subgroup $K_0$ of 
  $ \Bohr(\Ga^{\rm Ab})_0.$
  The  projection of $L$ on $ \Bohr(\Ga^{\rm Ab})_0$ induces then a continuous homomorphism
  $$
  f: L/L_0\to  \Bohr(\Ga^{\rm Ab})_0/K_0.
  $$
 Observe that $f$ has dense image, since $p\circ \beta_0(\Ga)$ is dense 
  in  $\Bohr(\Ga^{\rm Ab})_0$; so,  $f$ is surjective by compactness of $L/L_0.$
  It follows, by compactness again, that $\Bohr(\Ga^{\rm Ab})_0/K_0$ is topologically isomorphic
  to a quotient of $L/L_0.$ As  $L/L_0$ is totally disconnected, this  implies (see \cite[Chap. 3, \S 4, Corollaire 3]{Bourbaki}) that
   $\Bohr(\Ga^{\rm Ab})_0/K_0$ is also totally disconnected  and hence that $K_0=\Bohr(\Ga^{\rm Ab})_0.$
  So, $\Bohr(\Ga^{\rm Ab})_0\times \{1\}$ is contained in $L.$
  It follows that $L$ is the product of $ \Bohr(\Ga^{\rm Ab})_0$ with a   subgroup   of $\Prof(\Ga).$
  Since $\alpha(\Ga)$ is dense in $\Prof(\Ga),$ this subgroup coincides with $\Prof(\Ga)$, that is, 
  $L=K$ and the claim is proved.

  \vskip.2cm
  $\bullet$ {\it Third  step.} We claim  that  every  irreducible unitary representation  $\pi:\Ga\to U(N)$ of $\Ga$
 is of the form $\chi \otimes \rho$ for some $\chi\in\widehat{\Ga^{\rm Ab}}$ and $\rho\in\widehat{\Ga}_{\rm finite}.$
 
Indeed, Propositions~\ref{Pro-DistorsionBohr} and \ref{Prop-NilDist},
imply that $\pi([\Ga, \Ga])$ is a periodic subgroup of $U(N).$
Since $\Ga$ is finitely generated, $[\Ga, \Ga]$ is finitely generated (in fact, every subgroup 
of $\Ga$ is finitely generated; see \cite[2.7 Theorem]{Raghunathan}). Hence,
by Schur's theorem (see \cite[4.9 Corollary]{Wehrfritz}), $\pi([\Ga, \Ga])$ is finite.
It follows that  there exists a finite index normal subgroup $H$ of $[\Ga, \Ga]$ 
so that $\pi|_H$ is the trivial representation of $H$. 
 
Next, we claim that there exists a normal  subgroup $\Delta$ of finite index in $\Ga$ such that 
$\Delta \cap [\Ga, \Ga]=H.$ 
Indeed, since  $\Ga/[\Ga, \Ga]$ is abelian and finitely generated, we have
$\Ga/[\Ga, \Ga]\cong \ZZ^k\oplus F$ for  some finite subgroup  $F$ and some integer $k\geq 0.$
Let $\Ga_1$ be the inverse image in $\Ga$ of the copy of $\ZZ^k$ in  $\Ga/[\Ga, \Ga].$
Then $\Ga_1$ is a normal subgroup  of finite index of $\Ga.$ Moreover, $\Ga_1$ can be written
as  iterated semi-direct product
$$\Ga_1=(\dots (([\Ga, \Ga] \rtimes \ZZ)\rtimes \ZZ) \rtimes \ZZ)).$$
Set 
$$
\Delta:= (\dots ((H \rtimes \ZZ)\rtimes \ZZ) \rtimes \ZZ)).
$$
Then $\Delta$ is a  normal subgroup  of finite index of $\Ga$  with $\Delta \cap [\Ga, \Ga]=H.$ 
 
 Since $\pi|_H$ is trivial on $H$ and since $[\Delta, \Delta]\subset H,$ 
 the restriction $\pi|_{\Delta}$ of $\pi$ to $\Delta$ factorizes through $\Delta^{\rm Ab}.$
 So, by Proposition~\ref{Prop-MackeyMachine1}, there exists
 a finite $\Ga$-orbit $\mathcal{O}$ in $\widehat{\Delta^{\rm Ab}}$
such that we have a direct sum decomposition 
$V= \bigoplus_{\chi \in \mathcal{O}} V^{\chi}$, where $V^{\chi}$
is the $\chi$-isotypical component of $\pi|_\Delta.$ 

Fix $\chi\in \mathcal{O}.$ Since $\chi$ is trivial on $H$ and since $\Delta \cap [\Ga, \Ga]=H,$ 
we can view $\chi$ as a unitary character of the  subgroup $\Delta/ (\Delta \cap [\Ga, \Ga])$
of $\Ga^{\rm Ab}.$ Hence, $\chi$ extends to a character $\widetilde{\chi}\in \widehat{\Ga^{\rm Ab}}$
(see, e.g. \cite[(24.12)]{HewittRoss}).
This implies that $\Ga_\chi= \Ga$; indeed, 
$$\chi^\ga(\delta)= \widetilde{\chi}(\ga^{-1}\delta\ga)=   \widetilde{\chi}(\delta)=\chi(\delta)$$ 
for every $\ga\in \Ga$ and $\delta\in \Delta.$ This shows that $\mathcal{O}$ is a singleton 
and so $V=V^{\chi}.$ 
We write $$\pi=\widetilde{\chi}\otimes (\overline{\widetilde{\chi}}\otimes \pi).$$
Then $\rho:=\overline{\widetilde{\chi}}\otimes \pi$  is an irreducible unitary representation of $\Ga$ which is trivial on $\Delta$;
so,  $\rho$ has finite image and $\pi=\widetilde{\chi}\otimes\rho.$

\vskip.2cm
  $\bullet$ {\it Third  step.}  Let  $\pi\in  \widehat{\Ga}_{\rm fd}.$
  We claim that there exists a representation $\pi'\in \widehat{K}$
  such that $\pi= \pi'\circ \beta.$ Once proved, Proposition~\ref{GeneralBohr} will  imply that  $(K, \beta)$ is a Bohr compactification for $\Ga$.
  
  By the second step, we can write $\pi= \chi\otimes \rho$ for some
  $\chi\in\widehat{\Ga^{\rm Ab}}$ and $\rho\in\widehat{\Ga}_{\rm finite}.$
  On the one hand, we can write $\rho= \rho'\circ \alpha$ for some $\rho'\in \widehat{\Prof(\Ga)}$,
  by the universal property of $\Prof(\Ga).$
  On the other hand, we can decompose $\chi$ as $\chi= \chi_0 \chi_1$  
 with $\chi_0\in\widehat{\Ga^{\rm Ab}}$ of infinite order and $\chi_1\in \widehat{\Ga^{\rm Ab}}$ of finite order.
 We have  $\chi_0= \chi'_0 \circ (p\circ \beta_0)$ and $\chi_1= \chi_1'\circ \alpha$
 for unitary characters $\chi'_0$ of $\Bohr(\Ga^{\rm Ab})_0$ and $\chi_1'$
 of $\Prof(\Ga^{\rm Ab})$. For  $\pi'= \chi_0 \otimes (\chi'_1 \otimes \rho')$, we have
  $\pi'\in  \widehat{K}$ and $\pi= \pi'\circ \beta.$
 \section{Proof of Theorems~\ref{Theo-General}}
\label{Proof-Theo-General}
 Let $\GG= \UU\rtimes \HH $ be a Levi decomposition of $\GG$ and set 
  $$
  \Lambda=\HH(\ZZ), \quad \Delta=\UU(\ZZ), \qquad\text{and}\qquad  \Ga=\Delta\rtimes \Lambda.
  $$
  Denote by $\beta_\Delta: \Delta \to  \Bohr(\Delta)$ and $\beta_\Lambda: \Lambda \to  \Bohr(\Lambda)$ the natural homomorphisms.
  Observe   that, by the universal property of $\Bohr(\Delta),$ every element $\lambda\in \Lambda$ defines   a continuous automorphism $\theta_b(\lambda)$ of $\Bohr(\Delta) $ 
 such that 
 $$\theta_b(\lambda)(\delta)= \beta_\Delta( \lambda \delta \lambda^{-1})\tout \delta\in \Delta.$$ 
   The  corresponding homomorphism $\theta_b: \Lambda\to  \Aut(\Bohr(\Delta))$
   defines an action of $\Lambda$ on $\Bohr(\Delta).$
   By Theorem~\ref{Theo-Nilpotent}, we have
  $$\Bohr(\Delta)= \Bohr(\Delta^{\rm Ab})_0\times \Prof( \Delta).$$
  
  The group $\Lambda$ acts naturally on $\Delta^{\rm Ab}$ and, by duality, on  $\widehat{\Delta^{\rm Ab}}.$
  Let 
  $$H:=\widehat{\Delta^{\rm Ab}}_{\Lambda-\rm{fin}}\subset \widehat{\Delta^{\rm Ab}} $$  
  be  the subgroup  of  characters  of $\Delta^{\rm Ab}$ with finite $\Lambda$-orbits. Observe that $H$ contains   the torsion subgroup of $\widehat{\Delta^{\rm Ab}}.$  
 
 Let $$\alpha: \Lambda\to  \Aut(H)$$
   be the homomorphism given by the  action of $\Lambda$ on $H.$

   For a  locally compact group $G,$ the group  $\Aut(G)$ of continuous automorphisms of $G$
will be endowed with the compact-open topology for which it is also a (not necessarily locally 
compact)  topological group (see \cite[(26.3)]{HewittRoss}).

 \vskip.2cm
  $\bullet$ {\it First step.}  We claim that the closure   of $\alpha(\Lambda)$ in  $\Aut(H)$ is compact.
  Indeed,  let us identify  $\Aut(H)$ with a subset of the product space $H^H$. The topology 
 of $\Aut(H)$ coincides with  the topology induced by  the product topology on $H^H.$
 Viewed this way, $\alpha(\Lambda)$ is a  subspace of the product $\prod_{\chi\in H} \chi^\Lambda$
 of the finite $\Lambda$-orbits $\chi^\Lambda.$
 Since  $\prod_{\chi\in H} \chi^\Lambda$ is compact and hence closed, the claim is proved.

  \vskip.2cm
   Next, let $N$ be the annihilator of $H$ in $\Bohr(\Delta^{\rm Ab}).$
   Then $N$ is $\Lambda$-invariant and  the induced action of $\Lambda$  on $ \Bohr(\Delta^{\rm Ab})/N$
is  a quotient of the action given by $\theta_b.$
  
  Let $C$ be the connected component of $\Bohr(\Delta^{\rm Ab})/N.$
  Then $C$  coincides with  the image of $\Bohr(\Delta^{\rm Ab})_0$ in $\Bohr(\Delta^{\rm Ab})/N$ (see \cite[Chap. 3, \S 4, Corollaire 3]{Bourbaki}) and so 
  $$C\cong \Bohr(\Delta^{\rm Ab})_0/(N\cap \Bohr(\Delta^{\rm Ab})_0).$$
  Since $C$ is invariant under $\Lambda$, we obtain an action of $\Lambda$ on  $C;$ let 
$$\widehat{\alpha}: \Lambda \to \Aut(C)$$
be the corresponding homomorphism.

 \vskip.2cm
  $\bullet$ {\it Second step.} 
   We claim that the action $ \widehat{\alpha}$  of $\Lambda$ on $C$ 
extends to an action of $\Bohr(\Lambda)$; more precisely, there exists  a continuous homomorphism
$$ \widehat{\alpha}': \Bohr(\Lambda) \to \Aut(C)$$
such that the diagram
\[
\begin{tikzcd}
& \Bohr(\Lambda) \arrow[dashed]{d}{\widehat{\alpha}'} \\
\Lambda \arrow{ur}{\beta_\Lambda} \arrow{r}{\widehat{\alpha}} & \Aut(C)
\end{tikzcd}
\]
commutes.
Indeed, by the first step,  the closure $K$ of $\alpha(\Lambda)$ in  $\Aut(H)$ is a compact group.
Hence, by the universal property of $\Bohr(\Lambda),$ there exists a continuous homomorphism
$$\alpha':\Bohr(\Lambda) \to K\subset \Aut(H)$$ such that
the diagram 
\[
\begin{tikzcd}
& \Bohr(\Lambda) \arrow[dashed]{d}{\alpha'} \\
\Lambda \arrow{ur}{\beta_\Lambda} \arrow{r}{\alpha} &\Aut(H)
\end{tikzcd}
\]
commutes. Since $\widehat{H}=\Bohr(\Delta^{\rm Ab})/N,$   we obtain by duality a continuous homomorphism
$\widehat{\alpha}': \Bohr(\Lambda) \to \Aut(\Bohr(\Delta^{\rm Ab})/N)$.
The connected component $C$ of $\Bohr(\Delta^{\rm Ab})/N$ is 
invariant under $\Bohr(\Lambda)$.
This proves the existence of the map 
$\widehat{\alpha}': \Bohr(\Lambda) \to \Aut(C)$ with the claimed property.

\vskip.2cm
Next, observe  that, by the universal property of $\Prof(\Delta),$  every element $\lambda\in \Lambda$ defines   a continuous automorphism $\theta_p(\lambda)$ of $\Prof(\Delta) $ 
 such that 
  $$\theta_p(\lambda)(\delta)= \beta_\Delta( \lambda \delta \lambda^{-1})\tout \delta\in \Delta.$$ 
   The  corresponding homomorphism $\theta_p: \Lambda\to  \Aut(\Prof(\Delta))$
   defines an action of $\Lambda$ on $\Prof(\Delta).$

   \vskip.2cm
  $\bullet$ {\it Third  step.} 
   We claim that the action $ \theta_p$  of $\Lambda$ on  $\Prof(\Delta)$
   extends to an action of $\Bohr(\Lambda);$ more precisely, there exists a homomorphism $\theta': \Bohr(\Lambda) \to \Aut( \Prof( \Delta))$ 
such that the diagram
\[
\begin{tikzcd}
& \Prof(\Lambda) \arrow[dashed]{d}{\theta'} \\
\Lambda \arrow{ur}{\beta_\Lambda} \arrow{r}{ \theta_p} & \Aut(\Prof( \Delta))
\end{tikzcd}
\]
commutes. Indeed,  since $\Delta$ is finitely generated and since its image in  $\Bohr(\Delta)$ 
dense, the profinite group $\Bohr(\Delta)$ is finitely generated (that is, 
there exists a finite subset  of $\Bohr(\Delta)$ which generates a dense subgroup).
This implies that $\Aut(\Bohr(\Delta))$ is a profinite group (see \cite[Corollary 4.4.4]{Ribes-Zal})
and so there exists a homomorphism $\theta'_p: \Prof(\Lambda) \to \Aut( \Prof( \Delta))$ 
such that $\theta'_p\circ  \alpha_\Lambda= \theta_p.$ We then lift
$\theta'_p$ to a homomorphism $\theta': \Bohr(\Lambda) \to \Aut( \Prof( \Delta)).$

\vskip.2cm
We set $$Q:= \Bohr(\Delta)/(N\cap \Bohr(\Delta^{\rm Ab})_0)= C\times  \Prof( \Delta);$$
  we have an action of  $\Lambda$ on $Q$ given by the homomorphism
   $$\widehat{\alpha} \oplus \theta_p: \Lambda \to  \Aut(C)\times \Aut(\Prof(\Delta))\subset \Aut(Q)$$
  and, by the second and third step, an action of  $\Bohr(\Lambda)$ on $Q$
  given by $$\widehat{\alpha}' \oplus \theta':\Bohr(\Lambda) \to  \Aut(C)\times \Aut(\Prof(\Delta))$$
  such that   the diagram
  \[
\begin{tikzcd}
& \Bohr(\Lambda) \arrow[dashed]{d}{\widehat{\alpha}' \oplus \theta'} \\
\Lambda \arrow{ur}{\beta_\Lambda} \arrow[r, swap, "\widehat{\alpha}\oplus \theta_p"] &\Aut(C)\times \Aut(\Prof(\Delta))
\end{tikzcd}
\]
commutes.

  \vskip.2cm
    Let 
  $$B:= (C\times \Prof(\Delta)) \rtimes \Bohr(\Lambda)$$
  be the semi-direct product defined  by  $\widehat{\alpha}' \oplus \theta'.$
  Let 
  $$p:\Bohr(\Delta)\to  C= \Bohr(\Delta^{\rm Ab})_0/(N\cap \Bohr(\Delta^{\rm Ab})_0)$$
  be the quotient epimorphism. 
  
  \vskip.2cm
  $\bullet$ {\it Fourth step.}   We claim that $B,$ together with the map 
 $\beta:\Ga\to B,$
given by 
  \[
  \beta(\delta, \lambda)= (p(\beta_\Delta (\delta)),\beta_\Lambda(\lambda)) \tout (\delta, \lambda)\in \Ga,
  \]
 is a Bohr compactification for $\Ga=  \Delta\rtimes \Lambda.$ 
 
 First, we have to  check that $\beta$ is a homomorphism with dense image.
 Since $p\circ \beta_\Delta$ and $\beta_\Lambda$
 are homomorphisms with dense image, it suffices  to show that 
 $$\beta(\lambda \delta \lambda^{-1}, e)=((\widehat{\alpha}' \oplus \theta')(\beta_\La(\lambda)) (p(\beta_\Delta (\delta)),e)\tout (\delta, \lambda)\in \Ga.$$
 This is indeed the case: since $p$ is equivariant for the $\Lambda$-actions, we have 
 $$p(\beta_\Delta (\lambda \delta \lambda^{-1}))=p (\theta_b(\lambda)\beta_\Delta(\delta))= (\widehat{\alpha}' \oplus \theta')(\beta_\La(\lambda))p(\beta_\Delta (\delta)).$$
 
 Next, let  $\pi$ be a unitary representation
 of $\Ga$ on a finite dimensional vector space $V.$ By Proposition~\ref{GeneralBohr}, we have  to show that 
 there exists a  unitary representation $\widetilde{\pi}$ of $B$ on $V$ 
 such that $\pi= \widetilde{\pi} \circ \beta$.
 
 Consider a decomposition of $V= V_1\oplus \cdots \oplus V_s$ into irreducible
 $\pi(\Delta)$-invariant subspaces $V_i$; denote by $\sigma_1, \dots, \sigma_s$ the corresponding 
 irreducible representations of $\Delta.$
 By Theorem~\ref{Theo-Nilpotent}, every  $\sigma_i$ is of the form $\sigma_i=\chi_i \otimes \rho_i$
 for some $\chi_i\in\widehat{\Delta^{\rm Ab}}$ and $\rho_i\in\widehat{\Delta}_{\rm finite}.$

  We  decompose every $\chi_i$ as  a product $\chi_i=\chi_i' \chi_i ''$ 
 with $\chi_i'\in\widehat{\Delta^{\rm Ab}}$ of finite order and $\chi_i''\in\widehat{\Delta^{\rm Ab}}$ of infinite order.
 Since $\chi_i'$  has finite image,  upon replacing $\rho_i$
 by $\chi_i'\otimes \rho_i$, we may and will assume that  every non trivial  $\chi_i$ has infinite order.

Fix $i\in \{1, \dots, s\}.$   We can extend $\chi_i$ and $\rho_i$ to unitary representations of $\Bohr(\Delta),$ that is, we can
 find  representations  $\widetilde{\chi_i}$  and $\widetilde{\rho_i}$ of $\Bohr(\Delta)$ on  $V_i$
 such that 
 $\chi_i= \widetilde{\chi_i}\circ \beta_\Delta$ and $\rho_i= \widetilde{\rho_i}\circ  \beta_\Delta.$
By Proposition~\ref{Prop-MackeyMachine1}, the stabilizer $\Ga_{\sigma_i}$  of $\sigma_i$  has finite index
in $\Ga.$  It follows that the $\Lambda$-orbit of $\sigma_i$ is finite, and this implies that $\chi_i\in H;$
 hence, $\widetilde{\chi_i}$ factorizes through  
 $$C=\Bohr(\Delta^{\rm Ab})_0/(N\cap \Bohr(\Delta^{\rm Ab})_0)$$
 and we have $\chi_i=  \widetilde{\chi_i}\circ (p\circ \beta_\Delta).$
 Since  $\rho_i$ has finite image, $\widetilde{\rho_i}$ factorizes through
 $\Prof(\Delta)$. So, $\widetilde{\sigma_i}:=\widetilde{\chi_i} \otimes \widetilde{\rho_i}$ is a unitary representation of $C\times \Prof(\Delta)$ on $V_i$.
  Set 
 $$\widetilde{\pi_\Delta}:= \bigoplus_{i=1}^s \widetilde{\sigma_i}.$$
 Then  $\widetilde{\pi_\Delta}$ is a unitary representation of $C\times \Prof(\Delta)$ on $V$ 
 such that  $ \pi|_{\Delta}=\widetilde{\pi_\Delta}\circ (\beta|_\Delta).$
 
 On the other hand, since $\pi|_\La$ is a finite dimensional representation of $\La,$
 we can find a representation $\widetilde{\pi_\Lambda}$ of  $\Bohr(\La)$
 on $V$  such that $\pi|_{\Lambda}=\widetilde{\pi_\Lambda}\circ (\beta|_\La).$
 For $\lambda\in \Lambda$ and $\delta\in \Delta$, we have
 \[
 \begin{aligned}
 \widetilde{\pi_\Delta}(\beta(\lambda)\beta(\delta)\beta(\lambda)^{-1})&= \widetilde{\pi_\Delta}(\beta(\lambda\delta\lambda)^{-1}) \\
 & =\pi(\lambda\delta\lambda)^{-1})\\
 &=\pi(\lambda)\pi(\delta)\pi(\lambda)^{-1}\\
 &=\widetilde{\pi_\Lambda}(\beta(\lambda))  \widetilde{\pi_\Delta}(\beta(\delta)) \widetilde{\pi_\Lambda}(\beta(\lambda))^{-1}.
 \end{aligned}
\]
 Since $\beta$ has dense image in $B,$ it follows that 
 \[
  \widetilde{\pi_\Delta}(bab^{-1})=\widetilde{\pi_\Lambda}(b) \widetilde{\pi_\Delta}(a) \widetilde{\pi_\Lambda}(b)^{-1} \tout (a,b)\in B
  \]
and therefore the formula
\[
 \widetilde{\pi} (a, b)= \widetilde{\pi_\Delta}(a) \widetilde{\pi_\Lambda}(b) \tout (a,b)\in B
 \]
defines a unitary representation of $B$ on $V$  such that $ \pi=\widetilde{\pi}\circ \beta.$

\section{Proof of Theorem~\ref{Theo-Reductive}}
\label{S-Proof-Theo-Reductive}
Recall that  we are assuming that $\GG$ is a  connected, simply-connected and almost $\QQ$-simple algebraic group.
The group $\GG$  can be obtained from  an absolutely simple
 algebraic group  $\HH$   by the so-called restriction of scalars; more precisely
 (see \cite[6.21, (ii)]{Borel-Tits1}),  there exists a number field $K$ and an absolutely simple
 algebraic group $\HH$ over $K$ which is absolutely simple  with the following property:
 $\GG$ can be written as (more precisely, is $\QQ$-isomorphic to)
  the $\QQ$-group  $\HH^{\sigma_1} \times \dots \times \HH^{\sigma_s},$
 where the $\sigma_i$'s are the different (non conjugate) embeddings of $K$ in $\CCC$.
 Assuming that  $\sigma_1, \dots, \sigma_{r_1}$ are the embeddings such that $\sigma_i(K)\subset \RRR,$
 we can identify $\GG(\RRR)$ with
 $$\HH^{\sigma_1}(\RRR) \times \dots \times \HH^{\sigma_{r_1}}(\RRR) \times \HH^{\sigma_{r_1+1}}(\CCC) \times \dots \times \HH^{\sigma_{r_s}}(\CCC).$$
Let $\GG_{\rm c}$  be the product of the $\HH^{\sigma_i}$'s  for which $\HH^{\sigma_i}(\RRR)$
is   compact.
    
   We assume now that the real semisimple Lie  group $\GG(\RRR)$ is  not locally isomorphic  to a group of the form $SO(m, 1)\times L$ or $SU(m,1)\times L$ for a compact Lie group $L$.
  Let  $\Ga \subset \GG(\QQ)$ be an arithmetic subgroup. 

  Set $K:=\GG_{\rm c}(\RRR) \times \Prof(\Ga)$ and let $\beta: \Ga\to K$ be defined
  by $\beta(\ga)=(p(\ga), \alpha(\ga)),$ where $p: \GG(\RRR)\to\GG_{\rm c}(\RRR)$ is the canonical projection
 and $\alpha: \Ga\to \Prof(\Ga)$ the map associated to $\Prof(\Ga).$
 We claim that $(K, \beta)$ is a Bohr compactification of $\Ga$,
 
 First, we show that  $\beta(\Ga)$ has dense image. 
 Observe that $\GG_{\rm c}(\RRR)$ is connected (see \cite[(24.6.c)]{Borel-Book}).
 By the Strong Approximation Theorem (see \cite[Theorem 7.12]{Platonov}),
 $p(\GG(\ZZ))$ is dense in $\GG_{\rm c}(\RRR)$.
Since $\GG_{\rm c}(\RRR)$ is connected and since $\Ga$ is commensurable to $\GG(\ZZ)$, it follows that
$p(\Ga)$ is dense in $\GG_{\rm c}(\RRR)$.
Now, $\alpha(\Ga)$ is dense in $\Prof(\Ga)$ and $\Prof(\Ga)$ is totally disconnected.
As in the first step of the proof  of Theorem~\ref{Theo-Nilpotent}, we conclude that
 $\beta(\Ga)$ is dense in $K.$

   Let $\pi: \Ga\to U(n)$ be a finite dimensional unitary  representation of $\Ga.$
  Then, by Margulis' superrigidity theorem (see  \cite[Chap. VIII, Theorem B]{Margulis}), \cite[Corollary 16.4.1]{Witte}),
  there exists a continuous homomorphism $\rho_1: \GG(\RRR) \to U(n)$ and a homomorphism $\rho_2: \Ga\to  U(n)$ such that 
  \begin{itemize}
  \item[(i)] $\rho_2(\Ga)$ is finite;
  \item[(ii)] $\rho_1(g)\rho_2(\ga)= \rho_2(\ga)\rho_1(g)$ for all $g\in \GG(\RRR)$ and $\ga\in \Ga;$
  \item[(iii)] $\pi(\ga)= \rho_1(\ga)\rho_2(\ga)$ for all $\ga\in \Ga.$
  \end{itemize}
   By a classical result of  Segal and von Neumann \cite{Segal-vN}, $\rho_1$
   factorizes through $\GG_{\rm c}(\RRR)$, that is,  $\rho_1=\rho_1'\circ p$ for a unitary representation $\rho_1'$ of $\GG_{\rm c}(\RRR).$  It follows from (i) that $\rho_2= \rho'_2\circ \alpha$ for a unitary representation $\rho_2'$ of $\Prof(\Ga).$  Moreover, (ii) and (iii) show that $\pi= (\rho_1|_{\Ga}) \otimes \rho_2.$
 Hence, $\pi=  (\rho_1' \otimes \rho'_2) \circ  \otimes \beta$. We conclude by 
 Proposition~\ref{GeneralBohr} that $(K, \beta)$ is a Bohr compactification
 of $\Ga.$

 \section{A few examples}
 \label{S: Examples}
 We compute the Bohr compactification for various examples of arithmetic groups.
 \begin{enumerate}
 \item For an integer $n\geq 1$,  the $(2n+1)$-dimensional Heisenberg group is the unipotent 
 $\QQ$-group $\HH_{2n+1}$ of matrices of the form
 \[
 m(x_1, \dots, x_n, y_1, \dots, y_n, z):=
\left(
\begin{array}{cccccc}
 1&x_1&\dots &x_{n}&z\\
 0&1&\dots&0&y_1\\
 \vdots & \ddots &\ddots& \vdots&\vdots \\
 0&0&\dots&1&y_n\\
0&0&\dots&0&1\
\end{array}
\right).
\]
The arithmetic group $\Ga=\HH_{2n+1}(\ZZ)$ is nilpotent of step 2;
its commutator subgroup $[\Ga, \Ga]$  coincides with its center
$\{m(0, 0, z): z\in \ZZ\}$. So,  $\Ga^{\rm Ab}\cong \ZZ^{2n}$.
We have, by Theorem~\ref{Theo-Nilpotent},
$$
\Bohr(\Ga)\cong  \Bohr(\ZZ^{2n})_0 \times \Prof(\Ga)
$$
and hence, by Proposition~\ref {Prop-BohrAbelian} and Proposition~\ref{Pro-CSP-Nilp}
$$
\Bohr(\Ga)\cong (\prod_{ \omega \in \frak{c}} \AA/\QQ) \times 
\prod_{p \ \text{prime}} \HH_{2n+1}(\ZZ_p).
$$
 \item
 Let $\GG= SL_n$ for $n\geq 3$ or $\GG= Sp_{2n}$ for $n\geq 2.$
 Then $ SL_n(\ZZ)$ and $ Sp_{2n}(\ZZ)$ are non cocompact arithmetic  lattices in $SL_n(\RRR)$
 and $Sp_{2n}(\RRR),$ respectively.
 Hence, we have, by Corollary~\ref{Cor1},
$\Bohr(SL_n(\ZZ))= \Prof(SL_n(\ZZ))$ and $\Bohr( Sp_{2n}(\ZZ))= \Prof( Sp_{2n}(\ZZ))$.
Since $ SL_n(\ZZ)$ and $ Sp_{2n}(\ZZ)$ have the congruence subgroup property, it follows
that 
$$
\Bohr(SL_n(\ZZ))\cong \prod_{p \ \text{prime}} SL_{n}(\ZZ_p) \cong  SL_{n}(\Prof(\ZZ))
$$
and similarly 
$$
\Bohr(Sp_{2n}(\ZZ))\cong \prod_{p \ \text{prime}} Sp_{2n}(\ZZ_p) \cong Sp_{2n}(\Prof(\ZZ)).
$$
\item The group  $\Ga= SL_2(\ZZ[\sqrt{2}])$ embeds as a non cocompact arithmetic 
lattice of $SL_2(\RRR)\times SL_2(\RRR).$ So,   by Corollary~\ref{Cor1}, we have
$$\Bohr(SL_2(\ZZ[\sqrt{2}]))\cong \Prof(SL_2(\ZZ[\sqrt{2}])).$$ Moreover, since
$\Ga$ has the congruence subgroup property (see \cite[Corollaire 3]{Serre}),
it follows that 
$$\Bohr(SL_2(\ZZ[\sqrt{2}]))\cong\Cong(SL_2(\ZZ[\sqrt{2}])).$$
 \item  For $n\geq 4,$ consider the quadratic form 
 $$q(x_1,\dots,  x_n)= x_1^2+\dots +x_{n-1}^2- \sqrt{2} x_n^2-  \sqrt{2} x_{n+1}^2$$
 The  group $\GG= SO(q)$ of unimodular $(n+1)\times (n+1)$-matrices 
 which preserve $q$ is an almost simple algebraic group over the number field $\QQ[\sqrt{2}].$
 The subgroup $\Ga=SO(q,\ZZ[\sqrt{2}])$ of $\ZZ[\sqrt{2}]$-rational points 
 in $\GG$ embeds as a cocompact   lattice of  the  semisimple real Lie group $SO(n+1) \times SO(n-1, 2)$
 via the map 
 $$
 SO(q, \QQ[\sqrt{2}])\to SO(n+1) \times SO(n-1, 2) , \, \ga \mapsto (\ga^\sigma, \ga),
 $$
 where $\sigma$ is the  field automorphism of $\QQ[\sqrt{2}]$ given by $\sigma(\sqrt{2})=-\sqrt{2};$
 so, $SO(n+1) \times SO(n-1, 2)$ is the group  of real points of the $\QQ$-group 
 $R_{\QQ[\sqrt{2}]/\QQ}(\GG)$ obtained by restriction of scalars from the   $\QQ[\sqrt{2}]$-group $\GG$.
 Observe that $R_{\QQ[\sqrt{2}]/\QQ}(\GG)$ is  almost $\QQ[\sqrt{2}]$-simple  since $\GG$ is  almost 
 $\QQ$-simple. 
 By Theorem~\ref{Theo-Reductive}, we have
 $$
\Bohr(SO(q,\ZZ[\sqrt{2}]))\cong SO(n+1)\times \Prof(SO(q,\ZZ[\sqrt{2}]).
$$
\item   For $d\geq 2,$ let $D$ be a  central division algebra over $\QQ$  such that $D\otimes_\QQ \RRR$
is isomorphic to  the algebra $M_d( \RRR)$ of real $d\times d$-matrices.
There exists a subring $\mathcal O$ of $D$ which is a $\ZZ$-lattice in $D$ (a so-called
order in $D$).
There is an embedding $\varphi: D\to M_d( \RRR)$ such that 
$\varphi(SL_1(D) \subset SL_d(\QQ)$ and such that 
$\Ga:=\varphi(SL_1(\mathcal O)$ is an arithmetic cocompact lattice in $SL_d(\RRR),$
where $SL_1(D)$ is the group of norm one elements in $D;$
for all this, see \cite[\S 6.8.i]{Witte}. 
For $d\geq 2,$ we have
$$\Bohr(\Ga)\cong \Prof(\Ga).$$
So, this is an example of a \emph{cocompact} lattice $\Ga$ in a simple real Lie group
for which there exists no homomorphism $\Ga\to U(n)$ with infinite image;
the existence of such examples was mentioned in  \cite[(16.4.3)]{Witte}
\item For $n\geq 3,$ let $\Ga$ be the semi-direct  product 
 $\ZZ^n\rtimes SL_n(\ZZ),$ induced by the usual linear action of $SL_n(\ZZ)$ 
 on $\RRR^n.$ The dual action of $SL_n(\ZZ)$ on $\widehat{\ZZ^n}\cong  \RRR^n/\ZZ^n$
 is given by 
 $$
 SL_n(\ZZ)\times \RRR^n/\ZZ^n\to \RRR^n/\ZZ^n, \, (g, x+ \ZZ^n)\mapsto {}^tg x + \ZZ^n.
 $$
 It is well-known and easy to show that the subgroup
 of $SL_n(\ZZ)$-periodic orbits in $\widehat{\ZZ^n}$ corresponds to 
 $\QQ^n/ \ZZ^n$, that is, to the characters of finite image.
It follows from Theorem~\ref{Theo-General} that 
$$
\Bohr(\ZZ^n\rtimes SL_n(\ZZ))\cong \Bohr(SL_n(\ZZ))_0\times \Prof(\ZZ^n\rtimes SL_n(\ZZ)).
$$
For $n\geq 3,$ we have therefore
$$
\Bohr(\ZZ^n\rtimes SL_n(\ZZ))\cong  \Prof(\ZZ^n\rtimes SL_n(\ZZ)) \cong \prod_{p \ \text{prime}} \ZZ_p\rtimes  SL_{n}(\ZZ_p).
$$
 
  \end{enumerate}


\begin{thebibliography}{99}

\bib{Anzai-Kakutani1}{article}{
   author={Anzai, Hirotada},
   author={Kakutani, Shizuo},
   title={Bohr compactifications of a locally compact Abelian group. I},
   journal={Proc. Imp. Acad. Tokyo},
   volume={19},
   date={1943},
   pages={476--480},
}
	
\bib{BMS}{article}{
   author={Bass, H.},
   author={Milnor, J.},
   author={Serre, J.-P.},
   title={Solution of the congruence subgroup problem for ${\rm
   SL}_{n}\,(n\geq 3)$ and ${\rm Sp}_{2n}\,(n\geq 2)$},
   journal={Inst. Hautes \'{E}tudes Sci. Publ. Math.},
   number={33},
   date={1967},
   pages={59--137},
}


\bib{BHV}{book}{
   author={Bekka, B.},
   author={ Harpe, P. de la},
   author={Valette, A.},
   title={Kazhdan's property (T)},
   series={New Mathematical Monographs},
   volume={11},
   publisher={Cambridge University Press, Cambridge},
   date={2008},
}


\bib{BH}{book}{
   author={Bekka, B.},
   author={ Harpe, P. de la},
   title={Unitary representations of groups,
duals, and characters},
series={Graduate Studies in Mathematics},
   publisher={American Mathematical Society, Providence, RI},
}
\bib{Bohr1}{article}{
   author={Bohr, H.},
   title={Zur Theorie der fast periodischen Funktionen},
   language={German},
   journal={Acta Math.},
   volume={45},
   date={1925},
   number={1},
   pages={29--127},
}
\bib{Bohr2}{article}{
   author={Bohr, Harald},
   title={Zur Theorie der Fastperiodischen Funktionen},
   language={German},
   journal={Acta Math.},
   volume={46},
   date={1925},
   number={1-2},
   pages={101--214},
}

\bib{Borel-Book}{book}{
 AUTHOR = {Borel, A.},
     TITLE = {Linear algebraic groups},
    SERIES = {Graduate Texts in Mathematics},
    VOLUME = {126},
   EDITION = {Second},
 PUBLISHER = {Springer-Verlag, New York},
      YEAR = {1991},
     PAGES = {xii+288},
  }

\bib{Borel}{article}{
   author={Borel, A.},
   title={Some finiteness properties of adele groups over number fields},
   journal={Inst. Hautes \'{E}tudes Sci. Publ. Math.},
   number={16},
   date={1963},
   pages={5--30},
}

\bib{Borel-Tits1}{article}{
  AUTHOR = {Borel, A.},
   AUTHOR = {Tits, J.},
     TITLE = {Groupes r\'{e}ductifs},
   JOURNAL = {Inst. Hautes \'{E}tudes Sci. Publ. Math.},
    NUMBER = {27},
      YEAR = {1965},
     PAGES = {55--150},
     }
\bib{Borel-HC}{article}{
   author={Borel, A.},
   author={Harish-Chandra},
   title={Arithmetic subgroups of algebraic groups},
   journal={Ann. of Math. (2)},
   volume={75},
   date={1962},
   pages={485--535},
}

\bib{Bourbaki}{book}{
   author={Bourbaki, N.},
   title={\'{E}l\'{e}ments de math\'{e}matique. Topologie g\'{e}n\'{e}rale. Chapitres 1 \`a 4},
   publisher={Hermann, Paris},
   date={1971},
   pages={xv+357 pp. (not consecutively paged)},
}

\bib{Chevalley}{article}{
   author={Chevalley, C.},
   title={Deux th\'{e}or\`emes d'arithm\'{e}tique},
   language={French},
   journal={J. Math. Soc. Japan},
   volume={3},
   date={1951},
   pages={36--44},
}
	

\bib{Dixm--C*} {book}{
    author = {Dixmier, J.},
     title = {{$C\sp*$}-algebras},
 publisher = {North-Holland Publishing Co., Amsterdam-New York-Oxford},
      year= {1977},
}

\bib{Drutu}{book}{
   author={Dru\c{t}u, C.},
   author={Kapovich, M.},
   title={Geometric group theory},
   series={American Mathematical Society Colloquium Publications},
   volume={63},
   publisher={American Mathematical Society, Providence, RI},
   date={2018},
   pages={xx+819},
}
\bib{Franks}{article}{
   author={Franks, J.},
   author={Handel, M.},
   title={Distortion elements in group actions on surfaces},
   journal={Duke Math. J.},
   volume={131},
   date={2006},
   number={3},
   pages={441--468},
}
	
\bib{Gersten}{article}{
   author={Gersten, S. M.},
   author={Short, H. B.},
   title={Rational subgroups of biautomatic groups},
   journal={Ann. of Math. (2)},
   volume={134},
   date={1991},
   number={1},
   pages={125--158},
}
	\bib{Gromov}{book}{
   author={Ballmann, W.},
   author={Gromov, M.},
   author={Schroeder, V.},
   title={Manifolds of nonpositive curvature},
   series={Progress in Mathematics},
   volume={61},
   publisher={Birkh\"{a}user Boston, Inc., Boston, MA},
   date={1985},
   pages={vi+263},
   
   }
   \bib{Hart-Kunen}{article}{
   author={Hart, J. E.},
   author={Kunen, K.},
   title={Bohr compactifications of non-abelian groups},
   booktitle={Proceedings of the 16th Summer Conference on General Topology
   and its Applications (New York)},
   journal={Topology Proc.},
   volume={26},
   date={2001/02},
   number={2},
   pages={593--626},
}
\bib{HewittRoss}{book}{
   author={Hewitt, E.},
   author={Ross, K. A.},
   title={Abstract harmonic analysis. Vol. I},
   volume={115},
   edition={2},
   publisher={Springer-Verlag, Berlin-New York},
   date={1979},
   pages={ix+519},
}  


\bib{Holm}{article}{
   author={Holm, P.},
   title={On the Bohr compactification},
   journal={Math. Ann.},
   volume={156},
   date={1964},
   pages={34--46},
}

\bib{LMR}{article}{
   author={Lubotzky, A.},
   author={Mozes, S.},
   author={Raghunathan, M. S.},
   title={The word and Riemannian metrics on lattices of semisimple groups},
   journal={Inst. Hautes \'{E}tudes Sci. Publ. Math.},
   number={91},
   date={2000},
   pages={5--53 (2001)},
   }
	
	
	
\bib{Margulis}{book}{
   author={Margulis, G. A.},
   title={Discrete subgroups of semisimple Lie groups},
   publisher={Springer-Verlag, Berlin},
   date={1991},
   pages={x+388},
   review={\MR{1090825}},
}
    \bib{Mostow}{article}{
      AUTHOR = {Mostow, G. D.},
     TITLE = {Fully reducible subgroups of algebraic groups},
   JOURNAL = {Amer. J. Math.},
  FJOURNAL = {American Journal of Mathematics},
    VOLUME = {78},
      YEAR = {1956},
     PAGES = {200--221},
}

   
  
   \bib{VN}{article}{
   author={v. Neumann, J.},
   title={Almost periodic functions in a group. I},
   journal={Trans. Amer. Math. Soc.},
   volume={36},
   date={1934},
   number={3},
   pages={445--492},
   issn={0002-9947},
   review={\MR{1501752}},
}

\bib{Platonov}{book}{
AUTHOR = {Platonov, V. },
 AUTHOR = {Rapinchuk, A.},
    TITLE = {Algebraic groups and number theory},
    SERIES = {Pure and Applied Mathematics},
    VOLUME = {139},
PUBLISHER = {Academic Press, Inc., Boston, MA},
      YEAR = {1994},
    PAGES = {xii+614},
     }
\bib{Raghunathan-CSP}{article}{
   author={Raghunathan, M. S.},
   title={On the congruence subgroup problem},
   journal={Inst. Hautes \'{E}tudes Sci. Publ. Math.},
   number={46},
   date={1976},
   pages={107--161},
}

\bib{Raghunathan}{book}{
   author={Raghunathan, M. S.},
   title={Discrete subgroups of Lie groups},
   publisher={Springer-Verlag, New York-Heidelberg},
   date={1972},
}
			
\bib{Ribes-Zal}{book}{
   author={Ribes, L.},
   author={Zalesskii, P.},
   title={Profinite groups},
   volume={40},
   publisher={Springer-Verlag, Berlin},
   date={2000},
   pages={xiv+435},
}


\bib{Segal-vN}{article}{
   author={Segal, I. E.},
   author={von Neumann, J.},
   title={A theorem on unitary representations of semisimple Lie groups},
   journal={Ann. of Math. (2)},
   volume={52},
   date={1950},
   pages={509--517},
}
\bib{Serre}{article}{
   author={Serre, J.-P.},
   title={Le probl\`eme des groupes de congruence pour SL2},
   language={French},
   journal={Ann. of Math. (2)},
   volume={92},
   date={1970},
   pages={489--527},
}
		


\bib{Tits}{article}{
   author={Tits, J.},
   title={Free subgroups in linear groups},
   journal={J. Algebra},
   volume={20},
   date={1972},
   pages={250--270},
}	

\bib{Tits2}{article}{
   author={Tits, Jacques},
   title={Groupes \`a croissance polynomiale (d'apr\`es M. Gromov et al.)},
   conference={
      title={Bourbaki Seminar, Vol. 1980/81},
   },
   book={
      series={Lecture Notes in Math.},
      volume={901},
      publisher={Springer, Berlin-New York},
   },
   date={1981},
   pages={176--188},
}	

\bib{Wehrfritz}{book}{
   author={Wehrfritz, B. A. F.},
   title={Infinite linear groups. An account of the group-theoretic
   properties of infinite groups of matrices},
   publisher={Springer-Verlag, New York-Heidelberg},
   date={1973},
   pages={xiv+229},
}


\bib{Weil}{book}{
   author={Weil, A.},
   title={L'int\'{e}gration dans les groupes topologiques et ses applications},
   publisher={Hermann \& Cie, Paris},
   date={1940},
   pages={158},
}

\bib{Witte}{book}{
   author={Morris, D. Witte},
   title={Introduction to arithmetic groups},
   publisher={Deductive Press},
   date={2015},
   pages={xii+475},
}


\end{thebibliography}
\end{document}